\documentclass[reqno]{amsart}

\usepackage[T1]{fontenc} 
\usepackage{babel}

\usepackage{amsmath, amsthm, amssymb}
\usepackage[usenames,dvipsnames,svgnames,table]{xcolor}
\usepackage[colorlinks=true,allcolors=blue]{hyperref}
\usepackage{enumerate}
\usepackage{dsfont}
\usepackage{cite}
\usepackage{comment}
\usepackage{bbm}
\usepackage{todonotes}
\DeclareFontFamily{U}{mathx}{}
\DeclareFontShape{U}{mathx}{m}{n}{<-> mathx10}{}
\DeclareSymbolFont{mathx}{U}{mathx}{m}{n}
\DeclareMathAccent{\widecheck}{0}{mathx}{"71}

\usepackage[toc,page]{appendix}
\usepackage{verbatim}

\allowdisplaybreaks
\theoremstyle{plain}
\newtheorem{lemma}{Lemma}[section]
\newtheorem{theorem}[lemma]{Theorem}

\newtheorem{proposition}[lemma]{Proposition}

\theoremstyle{remark}
\newtheorem{remark}[lemma]{Remark}

\newcommand{\Eb}{\mathbb{E}}

\newcommand{\Var}{\mathrm{Var}}

\newcommand{\bx}{\boldsymbol{x}}

\newcommand{\bk}{\boldsymbol{k}}

\newcommand{\ba}{\boldsymbol{a}}

\newcommand{\bphi}{\boldsymbol{\phi}}

\makeatletter
\newcommand*{\rom}[1]{\expandafter\@slowromancap\romannumeral #1@}
\makeatother

\def\pp{\partial}

\numberwithin{equation}{section}

\begin{document}

%\vskip 0.125in

\title[Vicosities Estimation For Primitive Equations]
{Estimation of Anisotropic Viscosities for the Stochastic Primitive Equations}
%{Parameter estimation for the stochastically perturbed Primitive equations}

\date{First circulated: September 12, 2023. This version: March 12 2025.}
%\thanks{\textit{ }}

\author[I. Cialenco]{Igor Cialenco}
\address[I. Cialenco]
{	Department of Applied Mathematics, Illinois Institute of Technology, 10 W 32nd Str, John T. Rettaliata Engineering Center, Room 220, Chicago, IL 60616, USA.} \email{cialenco@iit.edu}

\author[R. Hu]{Ruimeng Hu}
\address[R. Hu]
{Department of Mathematics, Department of Statistics and Applied Probability, University of California, Santa Barbara, CA, 93106-3080, USA.} \email{rhu@ucsb.edu}

\author[Q. Lin]{Quyuan Lin*}\thanks{*Corresponding author. School of Mathematical and Statistical Sciences, Clemson University, Clemson, SC 29634, USA. E-mail address: quyuanl@clemson.edu}
\address[Q. Lin]
{	School of Mathematical and Statistical Sciences \\
Clemson University\\
Clemson, SC 29634, USA.} \email{quyuanl@clemson.edu}

\begin{abstract}
The viscosity parameters play a fundamental role in applications involving stochastic primitive equations (SPE), such as accurate weather predictions, climate modeling, and ocean current simulations. In this paper, we develop several novel estimators for the anisotropic viscosities in the SPE, using a finite number of Fourier modes of a single sample path observed within a finite time interval. The focus is on analyzing the consistency and asymptotic normality of these estimators. We consider a torus domain and treat strong, pathwise solutions in the presence of additive white noise (in time). Notably, the analysis for estimating horizontal and vertical viscosities differs due to the unique structure of the SPE and the fact that both parameters of interest are adjacent to the highest-order derivative. To the best of our knowledge, this is the first work addressing the estimation of anisotropic viscosities, with the potential applicability of the developed methodology to other models.

\end{abstract}

\maketitle



\section{Introduction}
The study of global weather prediction and climate dynamics relies heavily on the atmosphere and oceans. Ocean currents transport warm water from low latitudes to higher latitudes, where the heat can be released to the atmosphere to balance the Earth's temperature. A widely accepted model used to describe the motion and state of the atmosphere and ocean is the Boussinesq system, which combines the Navier–Stokes equations (NSE) with rotation and a heat (or salinity) transport equation. Due to the extraordinary organization and complexity of the flow in the atmosphere and ocean, the full governing equations appear to be challenging to analyze, at least for the foreseeable future. In particular, the global existence and uniqueness of smooth solutions to the $3D$ NSE is one of the most daunting mathematical problems. Fortunately, when studying oceanic and atmospheric dynamics at the planetary scale, the vertical scale (a few kilometers for the ocean, 10-20 kilometers for the atmosphere) is much smaller than the horizontal scale (thousands of kilometers). As a result, the large-scale ocean and atmosphere satisfy the hydrostatic balance based on scale analysis, meteorological observations, and historical data. Therefore, the primitive equations (PE), also known as the hydrostatic Navier-Stokes equations, are derived as the asymptotic limit of the small aspect ratio between the vertical and horizontal length scales from the Boussinesq system\cite{azerad2001mathematical,li2019primitive,li2022primitive,furukawa2020rigorous}. Due to their impressive accuracy, the following $3D$ viscous PE is a widely used model in geophysics (see, e.g., \cite{blumen1972geostrophic,gill1976adjustment,gill1982atmosphere,hermann1993energetics,holton1973introduction,kuo1997time,plougonven2005lagrangian,rossby1938mutual}):
\\

\begin{subequations}\label{PE-system}
\begin{align}
    d V + (V\cdot \nabla_h V + w\pp_z V -\nu_h \Delta_h V - \nu_z \pp_{zz}V + f_0 V^\perp + \nabla_h p) dt & = \sigma dW, 
    \\
    \pp_z p &  = 0,
    \\
    \nabla_h \cdot V + \pp_z w & = 0,
\end{align}
\end{subequations}
with initial condition $V(0)=V_0.$ Here the horizontal velocity $V = (u, v)$, vertical velocity $w$, and the pressure $p$ are functions of time and space $(t, x, y, z)=(t,\bx)$. The $2D$ horizontal gradient and Laplacian are denoted by $\nabla_h = (\partial_{x}, \partial_{y})$ and $\Delta_h = \partial_{xx} + \partial_{yy}$, respectively. The nonnegative constants $\nu_h, \nu_z$, are the horizontal viscosity, and the vertical viscosity, respectively. The parameter $f_0 \in \mathbb{R}$ stands for the Coriolis parameter, and we use the notation $V^\perp = (-v,u)$. The noise term $\sigma dW$ represents the external stochastic forcing,  which is rigorously defined below in  \eqref{eqn:noise}. For simplicity, we drop the temperature and salinity in the original primitive system, but the full model can be studied similarly, albeit with more tedious details and computations.

We study system \eqref{PE-system} in the torus $\mathbb T^3$, subject to the following boundary conditions:
\begin{gather*}
	V, w, p \text{ are periodic in } (x,y,z) \text{ with period } 2\pi,  \\
V, p \text{ are even in $z$, and } w \text{ is odd in } z.
\end{gather*}
Note that the symmetry condition in $z$ variable is invariant under the dynamics of system \eqref{PE-system}.
Here $w$ can be written as
$
w(x,y,z) = -\int_0^z \nabla_h \cdot u(x,y,\tilde{z})d\tilde{z}.
$

The viscosity parameter plays a fundamental role in applications involving \eqref{PE-system}, such as the accurate prediction of weather patterns, climate models, and ocean currents. In the literature, depending on whether the system has horizontal or vertical viscosity, the following four main models are considered: 
\begin{enumerate}
    \item PE with full viscosity, i.e., $\nu_h>0, \nu_z>0$. For the deterministic case, global well-posedness of strong solutions in Sobolev spaces was established in \cite{cao2007global,kobelkov2006existence,kukavica2007regularity,hieber2016global}. The well-posedness of the stochastic version was investigated in \cite{glatt2008stochastic,glatt2011pathwise,brzezniak2021well,debussche2011local,debussche2012global,agresti2022stochastic-1,agresti2022stochastic-2}.

    \item PE with only horizontal viscosity, i.e., $\nu_h>0, \nu_z=0$. In \cite{cao2016global,cao2017strong,cao2020global}, the authors consider deterministic horizontally viscous PE with anisotropic diffusivity and establish global well-posedness. The global well-posedness of the stochastic model is studied in \cite{saal2023stochastic}.
    
    \item PE with only vertical viscosity, i.e., $\nu_h=0, \nu_z>0$. Without horizontal viscosity, the deterministic PE is shown to be ill-posed in Sobolev spaces \cite{renardy2009ill}. To achieve well-posedness, one can consider some additional weak dissipation \cite{cao2020well}, assume that the initial data have Gevrey regularity and be convex \cite{gerard2020well}, or be analytic in the horizontal direction \cite{paicu2020hydrostatic,lin2022effect}. It is worth mentioning that whether smooth solutions exist globally or form singularities in finite time is still an open question. On the other hand, the last two authors of this paper studied the stochastic model and established local well-posedness with either analytic initial data \cite{hu2022local} or Sobolev initial data with convex condition \cite{hu2023pathwise}.

    \item Inviscid PE, i.e., $\nu_h=0, \nu_z=0$. The deterministic inviscid PE is ill-posed in Sobolev spaces \cite{renardy2009ill,han2016ill,ibrahim2021finite}. Moreover, smooth solutions of the inviscid PE can form singularity in finite time \cite{cao2015finite,wong2015blowup,ibrahim2021finite,collot2023stable}. On the other hand, with either some special structures (local Rayleigh condition) on the initial data in $2D$, or real analyticity in all directions for general initial data in both $2D$ and $3D$, the local well-posedness can be achieved \cite{brenier1999homogeneous,brenier2003remarks,ghoul2022effect,grenier1999derivation,kukavica2011local,kukavica2014local,masmoudi2012h}. Similar local well-posedness results hold for the stochastic case \cite{hu2022local,hu2023pathwise}.
\end{enumerate}

The development of statistical methods for estimating the parameters in a model has two significant practical implications. Firstly, when we believe that the considered family of models accurately describes the underlying physical phenomena, but the modeler lacks complete knowledge of the specific physical parameters within this family, a \emph{statistical estimator} is a tool for identifying these unknown parameters. Secondly, when the observer already has a priori information about the physical quantities in the model but harbors doubts about the overall model's validity, an \emph{estimator}, along with its asymptotic properties, serves as the initial tool for testing and validating the underlying model.

In this paper, we establish several estimators for the viscosity parameters $\nu_h, \nu_z$ in \eqref{PE-system}, assuming that all other model parameters are known.  We assume that the observations are performed in the Fourier space, and the observer has access to a  single sample path of a  finite collection of Fourier modes of the solution observed over a finite time interval $[0,T]$. The focus is on analyzing the consistency and asymptotic normality of these estimators. We emphasize that generally speaking, in PE $\nu_h\neq \nu_z$, which, as we will see, makes the statistical problem quite different from those studied in the existing literature. To the best of our knowledge, this is the first work to address the estimation of anisotropic viscosities. We believe that this methodology may serve as a foundation for tackling inference problems involving anisotropic parameters in other models. 
The proposed estimators lead to the following main results.

{\noindent\bf Theorem 1.} [Consistency of estimators; Theorem~\ref{thm:consistency}] Let $\nu^N_{hj}$, $j=1,2,3$ be the estimators for $\nu_h$, and $\nu^N_{zj}$ and $\widehat{\nu^N_{zj}}$, $j=1,2,3$ be the estimators for $\nu_z$ described in Section~\ref{sec:derivation}.  Under some suitable assumptions, $\nu^N_{h1}$, $\nu^N_{z1}$, and $\widehat{\nu^N_{z1}}$ are weakly consistent estimators of the true parameters $\nu_h$ and $\nu_z$,  where $N$ is the number of Fourier modes of a single sample path observed within a finite time interval. That is,
    \[
    \lim\limits_{N\to\infty} \nu^N_{h1} = \nu_h, \quad \lim\limits_{N\to\infty} \nu^N_{z1} = \lim\limits_{N\to\infty} \widehat{\nu^N_{z1}} = \nu_z, \quad \text{in probability.}
    \] 
    Under further technical assumptions, $\nu^N_{h2}$ and $\nu^N_{h3}$ are weakly consistent estimators of the true parameters $\nu_h$, and $\nu^N_{z2}$ , $\widehat{\nu_{z2}^N}$, $\nu^N_{z3}$, $\widehat{\nu_{z3}^N}$ are weakly consistent estimators of the true parameters $\nu_z$.

{\noindent\bf Theorem 2.} [Asymptotic normality, Theorem~\ref{thm:normality}]
    Under proper assumptions, $\nu^N_{h1}$ and $\widehat{\nu_{z1}^N}$ are jointly asymptotically normal with rate $N^2$, i.e., 
    \begin{equation*}
         N^2 \left(
        \begin{array}{c}
        \nu^N_{h1}-\nu_h\\
        \widehat{\nu^N_{z1}}-\nu_z
        \end{array}
        \right) \stackrel{\mathcal{D}} \longrightarrow \Xi,
    \end{equation*}
where $\Xi$ is a two-dimensional normal random variable with mean zero and 
some explicit covariance matrix $\Sigma$.

By leveraging the results of joint asymptotic normality, one can construct asymptotic confidence intervals for any linear combination of $\nu_h$ and $\nu_z$. For instance, a $(1-\alpha)$-confidence interval for $\nu_h$ is given by $[\nu_{h1}^N - z_{\alpha/2}\sqrt{\Sigma_{11}}/N, \nu_{h1}^N + z_{\alpha/2}
\sqrt{\Sigma_{11}}/N]$, where $\nu_{h1}^N$ is the estimator we proposed for $\nu_h$,  
%$\Sigma$ represents the variance of $\Xi$
and $z_{\alpha/2}$ denotes the upper $\alpha/2$-quantile of a standard normal.

{\noindent \bf Related literature.} Statistical analysis of SPDEs is a relatively new research field with many recent developments, mostly for linear equations. Our work falls within the category of the so-called spectral methods, where the observations are done in the Fourier space and continuously in time over some finite time interval. This sampling scheme is one of the most widely studied in the literature, and for a comprehensive understanding of this classical method and its historical developments, we direct the readers to the survey \cite{cialenco2018statistical}, as well as \cite{CialencoDelgado-VencesKim2019,CialencoKimLototsky2019,DelgadoVencesPavonEspanol2023}. Recent developments encompass alternative sampling methods and inference techniques such as: estimation using local measurements \cite{altmeyer2021nonparametric,altmeyer2023parameter,JanakReiss2023}; assuming discrete time/space observations in the physical domain \cite{hildebrandt2021parameter,cialenco2020note,cialenco2023statistical,CialencoKim2020,AssaadTudor2021,KainoMasayukiUchida2019,AssaadEtAl2022,GamainTudor2023,TonakiEtAl2023,TonakiEtAl2023a}; methods involving data assimilation \cite{cotter2019numerically,nusken2019state,PasemannEtAl2023}; Bayesian inference \cite{reich2020posterior,yan2020bayesian,ZCG2018}.

A special place in the statistical analysis of SPDEs takes the nonlinear equations. The first attempt traces back to \cite{cialenco2011parameter}, where the authors explored the spectral approach and Maximum Likelihood formalism to study the estimation of the viscosity coefficient for the 2D Navier-Stokes equations. These ideas were adapted and extended to general reaction-diffusion systems in \cite{pasemann2020drift,Pasemann2021}. Similarly, using the same analytical tools from SPDEs, in \cite{altmeyer2023parameter}, the authors study reaction-diffusion systems in the realm of local measurements. Se also \cite{HildebrandtTrabs2021,Gaudlitz2023,GaudlitzReiss2023,AssaadTudor2021}. A fundamental step in dealing with nonlinear terms is the well-known PDE method of splitting the solution in the linear and nonlinear component, and using slightly better regularity of the nonlinear terms to prove that the nonlinear terms entering the estimator vanish as the flow of information increases. Similar to the general theory for PDEs or SPDEs, the use of fine properties of the solutions, such as regularity, for interesting and practically important SPDEs is done case by case, exploring the particular structure of the underlying equations. Primitive equations, studied in this work, fall in this class of models. Although we take a similar direction as in the spectral approach, the considered parameter estimation problems can not be treated directly by the existing results. We emphasize that in contrast to the existing literature, we aim to estimate two different parameters,  $\nu_h$ (the horizontal viscosity) and $\nu_z$ (the vertical viscosity), both next to the highest order derivatives. 
We employ an additional decomposition of the solution in its barotropic and baroclinic components, which allows us to separate $\nu_h$ in a single equation, and hence construct and study MLE type estimators. We show that these estimators are weakly consistent and asymptotically normal. However, such analysis can not be extended to $\nu_z$, in particular in establishing the rate of convergence. We propose a novel estimator $\widehat \nu_{z1}^N$ for $\nu_z$, and prove its asymptotic normality.  

We note that in the case of continuous-time observations, the parameter $\sigma$ can be determined exactly using standard quadratic variation arguments, hence assumed known without loss of generality. We also assume that all other parameters, $f_0$ and $\gamma$ (defined in the noise $\sigma$; see \eqref{eqn:noise} below), are known. Estimating these parameters constitutes a separate and distinct statistical problem, which is beyond the scope of this work.

Finally, we mention that in the context of deterministic PDEs, related works include, but are not limited to, the parameter recovery via data assimilation \cite{carlson2020parameter,carlson2021dynamically,martinez2022convergence}, and inverse problems for PDEs \cite{isakov2006inverse}.

The rest of the paper is organized as follows. In Section~\ref{sec:preliminary}, we introduce the notations and reformulate the original system~\eqref{PE-system}. In Section~\ref{sec:linear}, we analyze the associated linear system, providing various SPDE estimates and regularity results. Section~\ref{sec:derivation} is devoted to deriving the estimators and formulating the main results regarding their asymptotic properties. We study the regularity of the solution to \eqref{PE-system} in Section~\ref{sec:regularity}, and the proof of the main results is presented in Section~\ref{sec:proof}. In the Appendix, we collect some auxiliary technical results and, for the reader's convenience, recall some limit theorems from stochastic analysis.

\section{Preliminaries}\label{sec:preliminary}

In this section, we introduce notations and present some necessary preliminary results.  The universal constants $c$ and $C$ appearing below may change from line to line. When needed, we use subscripts to indicate the dependence of the constant on certain parameters, {\it e.g.}, we write $C_r$ to emphasize that the constant depends on $r$. For sequences $\{a_n\}_{n \geq 1}$ and $\{b_n\}_{n \geq 1}$, the notation $a_n\sim b_n$ means that $\lim\limits_{n\to \infty} \frac{a_n}{b_n}=C\neq 0$ if the limit exists, or $ c\leq \liminf\limits_{n\to \infty} \frac{a_n}{b_n} \leq  \limsup\limits_{n\to \infty} \frac{a_n}{b_n} \leq C$ for some constants $C < \infty$ and $c>0$ if the limit does not exist. If $\lim\limits_{n\to \infty} \frac{a_n}{b_n}=1$, we write $a_n\asymp b_n$. For $a,b\in\mathbb{R}$, the notation $a\lesssim b$ means that $a\leq Cb$ for some positive constant $C>0$. The notation $\mathbb{T}^3:=\mathbb R^3/2\pi\mathbb Z^3$ stands for the three-dimensional torus. We also denote by $z^*$ the complex conjugate of $z$. The notation $N \gg 1$ will be frequently employed when computing asymptotic orders, meaning that the computation is valid for values of $N$ much larger than 1. Boldface letters will be used to denote vectors, e.g. $\bx\in\mathbb{T}^3$. 

\subsection{Functional settings}
Let $L^2(\mathbb T^3)$ be the usual $L^2$ functional space consisting of square Lebesgue integrable $2\pi$-periodic functions, equivalently identified with $2\pi$-periodic functions on $\mathbb{R}^3$, endowed with norm  $\|\varphi\| = \|\varphi\|_{L^2} = \left(\int_{\mathbb T^3} |\varphi|^2 d\bx\right)^{\frac12}$. For a function $f$, we define its barotropic mode $\overline f$, and baroclinic mode $\widetilde f$, respectively, by 
\[
\overline f = \frac{1}{2\pi}\int_0^{2\pi} f(x,y,\tilde{z}) d\tilde{z}, \ \text{ and } \ \widetilde f= f-\overline{f}, \text{ respectively.}
\]
For $\bk \equiv (k_1,k_2,k_3)=(\bk',k_3)\in \mathbb{Z}^3$, we put
\begin{equation}
\phi_{\bk} =
\begin{cases}
\sqrt{2}e^{i\left( k_1 x_1 + k_2 x_2 \right)}\cos(k_3 z) & \text{if} \;  k_3\neq 0,\\
e^{i\left( k_1 x_1 + k_2 x_2 \right)} & \text{if} \;  k_3=0, \label{phik}
\end{cases}
\end{equation}
and define the space $H$ by
\begin{align*}
    H := \Bigg\{ \bphi \in L^2(\mathbb{T}^3) \; : \; \bphi= \sum\limits_{\bk\in \mathbb{Z}^3 ,\bk\neq 0} \ba_{\bk} \phi_{\bk}, \; &\ba_{-k_1, -k_2,k_3}=\ba_{k_1,k_2,k_3}^{*}, \; 
    \ba_{\bk}\cdot \bk' = 0 \text{ when $k_3=0$} \Bigg\}.
\end{align*}
Then $H$ is a closed subspace of $L^2(\mathbb T^3)$. In particular, the functions $f\in H$ are real-valued, even in $z$, have spatial zero means and satisfy $\nabla_h \cdot \overline f=0$. For our system \eqref{PE-system},  we have $V\in H$. 

For each $N\in\mathbb N$, let $H_N = \text{Span} \{\phi_{\bk}: |{\bk}|\leq N  \}$ be the finite dimensional subspace of $H$. Denote by $P_N$ the projection from $H$ to $H_N$, and by $\mathcal P_h$ the $2D$ Leray projection such that $\mathcal P_h \overline \varphi = \overline{\varphi} - \nabla_h  \Delta_h^{-1} \nabla_h \cdot \overline{\varphi}$, where $\overline{\varphi}$ is a $2D$ vector (see, for example \cite{constantin1988navier}). Here $ \Delta_h^{-1} $ represents the inverse of the Laplacian operator in $ \mathbb T^2 $ with zero mean value. Furthermore, let $\mathcal P$ be the hydrostatic Leray projection, that is $\mathcal P \varphi = \mathcal P_h \overline \varphi + \tilde \varphi$.

Denote by $w(u) = -\int_0^z \nabla_h \cdot u (x,y,\tilde{z})d\tilde{z}$, and define the nonlinear term 
$$B(f,g)=f\cdot \nabla_h g + w(f) \pp_z g.$$
Let $A= -\mathcal P \Delta$ be the hydrostatic Stokes operator (cf. \cite{giga2017bounded}). Moreover, we write $A_h = -\Delta_h$ and $A_z = - \partial_{zz}.$ For a given $\alpha \geq 0$, let 
$$\mathcal D(A^\alpha) = \left\{f\in H: \sum\limits_{\bk\in\mathbb Z^3,\bk\neq 0} |\bk|^{4\alpha} |f_{\bk}|^2 <\infty \right\},$$ 
where $f_{\bk} = \langle f ,\phi_{\bk}\rangle := \int_{\mathbb T^3} f \phi_{\bk}^* d\bx.$ 
As in the periodic case, $\mathcal P_h$ commutes with $\Delta_h$, and for $f\in \mathcal D(A)$ we have
\[
Af = -\mathcal P \Delta f = - \mathcal P_h \Delta_h \overline f - \Delta \widetilde f = - \Delta_h \overline f - \Delta \widetilde f = -\Delta f = (A_h+A_z) f.
\]

Next, we describe the stochastic term $\sigma dW$ in \eqref{PE-system}. Let $(\Omega, \mathcal{F}, \mathbb{F}, \mathbb{P})$ be a stochastic basis with a filtration $\mathbb{F} = (\mathcal{F}_t)_{t \geq 0}$ that supports a sequence of independent Brownian motions $\{W_{\bk}\}_{\bk \in \mathbb{Z}^3}$. We can formally represent $W$ as $W = \sum_{\bk \in \mathbb{Z}^3, \bk \neq 0} \phi_{\bk} W_{\bk}$. Let $L_2(H_1, H_2)$ denote the collection of Hilbert–Schmidt operators from $H_1$ to $H_2$. Throughout this work, we view $\sigma$ as an operator in $L_2(H,H)$, and consider the following additive noise:
\begin{equation}\label{eqn:noise}
    \sigma dW = \sigma_0\sum\limits_{\bk \in \mathbb Z^3, \bk\neq 0} |{\bk}|^{-\gamma} c_{{\bk}}\phi_{\bk} dW_{\bk},
\end{equation}
with $\gamma>\frac32$, where $\sigma_0>0$ is a fixed positive constant. Here $c_{\bk}\in \mathbb R^2$ with $|c_{\bk}|=1$, and $c_{\bk} = \frac{\bk'^{\perp}}{|\bk|}$ when $k_3=0$. Then, clearly $\sigma\in L_2(H, \mathcal D(A^{\frac\gamma2-\frac34-\varepsilon}))$, for any $\varepsilon>0$. Notably, when $k_3=0$, we have $c_{\bk}\cdot \bk' = 0$, resulting in  $\nabla_h \cdot \overline{c_{\bk}\phi_{\bk}}=0$. We highlight that this choice of noise ensures that the condition $\int_{\mathbb T^3} V d\bx = 0$ holds true. 

\subsection{Reformulation of \eqref{PE-system}}
To derive estimators  of the two parameters of interest  $\nu_h$ and $\nu_z$ in \eqref{PE-system}, and to prove their statistical properties, it will be helpful to rewrite \eqref{PE-system} so that one equation contains only one parameter. 
For this purpose, a natural approach is the barotropic and baroclinic decomposition, which is a classical technique in the analysis of PE (see, for example, \cite{cao2007global}). Consequently, we decompose $V$ into the barotropic part $\overline{V}$ and the baroclinic part $\widetilde{V}$. By applying $\mathcal P_h$ to $\overline V$, and noting that $\mathcal P_h$ commutes with $\Delta_h$ in the periodic setting, we rewrite the system \eqref{PE-system} as 
\begin{subequations}\label{PE-original}
\begin{align}
    &d \overline{V} + \left( \mathcal P_h\overline{B(V,V)} + \nu_h A_h \overline{V}\right)dt
    = \sigma_0\sum\limits_{k_3=0,\bk\neq 0} |{\bk}|^{-\gamma} c_{\bk}\phi_{\bk} dW_{\bk}, \label{PE-barotropic}
    \\
    &d \widetilde V + \left(\widetilde{B(V,V)} + \nu_h A_h \widetilde V + \nu_z A_z \widetilde V + f_0 \widetilde V^\perp \right)dt = \sigma_0\sum\limits_{k_3\neq 0} |{\bk}|^{-\gamma} c_{\bk}\phi_{\bk} dW_{\bk}. \label{PE-baroclinic}
\end{align}
\end{subequations}
Note that the rotation term $f_0 \overline{V}^\perp$ is part only of the pressure gradient, thus vanishes under Leray projection. Additionally, as $A_z \overline V=0$, we obtain \eqref{PE-barotropic} that contains only $\nu_h$.

In Section~\ref{sec:derivation}, we begin by deriving estimators for $\nu_h$ using \eqref{PE-barotropic}, and subsequently, we substitute these estimators into \eqref{PE-baroclinic} to obtain estimators for $\nu_z$. This approach leverages all available information (i.e., all Fourier modes) and yields consistent estimators for both $\nu_h$ and $\nu_z$. However, while we can show that the estimator for $\nu_h$ is asymptotically normal, we were unable to do so for $\nu_z$ (for further discussion, see Remark~\ref{rmk:nolimit}). To address this, we introduce a novel projection: for a fixed positive rational number $q\in\mathbb Q$ and for a generic function $f= \sum\limits_{\bk\in\mathbb Z^3, \bk\neq 0} f_{\bk} \phi_{\bk}$, we define the projection $\widehat f$ by
$$
\widehat f = \sum\limits_{|\bk'|=\sqrt{q}|k_3|\neq 0} f_{\bk} \phi_{\bk} \equiv \sum\limits_{|\bk'|^2=q|k_3|^2\neq 0} f_{\bk} \phi_{\bk}.
$$
Immediately, by writing
$
\widehat V = \sum\limits_{|\bk'|=\sqrt{q}|k_3|\neq 0} V_{\bk} \phi_{\bk},
$
we deduce the corresponding equation for $\widehat V$:
\[
d \widehat V + \left(\widehat{B(V,V)} + \nu_h A_h \widehat V + \nu_z A_z \widehat V + f_0 \widehat{V}^\perp\right)dt
     = \sigma_0\sum\limits_{|\bk'|=\sqrt{q}|k_3|\neq 0} |{\bk}|^{-\gamma} c_{\bk}\phi_{\bk} dW_{\bk}.
\]
The benefit of this projection is that $A_h \widehat f = q A_z \widehat f$. Therefore, the above equation can be written as:
\begin{equation*}
   d \widehat V + \left(\widehat{B(V,V)} + (\nu_h + \frac1q \nu_z) A_h \widehat V + f_0 \widehat{V}^\perp\right)dt
     =\sigma_0\sum\limits_{|\bk'|=\sqrt{q}|k_3|\neq 0} |{\bk}|^{-\gamma} c_{\bk}\phi_{\bk} dW_{\bk}. 
\end{equation*}

Mimicking the approach used for $\nu_h$ via \eqref{PE-barotropic}, one might attempt to isolate $\nu_z$ by computing the horizontal average of the original system \eqref{PE-system}. However, this method is not effective, as explained further in Remark~\ref{rmk:hor-ave-not-work}.  The main reason is that in the $3D$ case, such a projection reduces the dimensions by two, unlike the one-dimensional reduction for $\overline V$. This loss of two dimensions significantly impacts the estimator, causing a breakdown in the proof of Lemma~\ref{lemma:ratio-1}.

\section{Analysis of the Linear System}\label{sec:linear}
One key idea in proving the asymptotic properties of the derived estimators relies on the so-called splitting argument often used in the general theory of nonlinear PDEs. That is, decomposing the solution in its linear part that solves the corresponding linear equation and the nonlinear residual. Subsequently, we break down the estimators into elements that correspond to the linear and nonlinear parts of the solution. Analysis of each part fundamentally relies on the exact order of continuity of the linear and nonlinear components.  In this section, we present some analytical properties of the linear part of the solution, while the nonlinear part and the entire solution are studied in Section~\ref{sec:regularity}. 

Consider the linear system associated with \eqref{PE-original}: 
\begin{align}\label{PE-original-linear}
    &d U +  (\nu_h A_h U  + \nu_z A_z U + f_0 \mathcal P U^{\perp} ) dt 
    = \sigma_0\sum\limits_{\bk\neq 0} |{\bk}|^{-\gamma} c_{\bk}\phi_{\bk} dW_{\bk},
\end{align}
with the given initial condition $U(0)$.
Let $U_{\bk} = \langle U, \phi_{\bk}\rangle$ be the Fourier coefficients of the solution $U$. In addition,  we denote by $\overline{U}_{\bk} = U_{\bk}$, when $k_3=0$, and  by $\widetilde{U}_{\bk} = U_{\bk}$ when $k_3\neq 0$.
Note that $U_{\bk}\in \mathbb C^2$ is an Ornstein–Uhlenbeck process with the dynamics
\begin{align}
    &d \overline{U}_{\bk} + \nu_h |\bk|^2 \overline{U}_{\bk} dt = \sigma_0|{\bk}|^{-\gamma} c_{\bk} dW_{\bk}, \text{ \, when\, } k_3=0, \label{eq:barU}
    \\
    &d \widetilde{U}_{\bk}+ \Big( (\nu_h |\bk'|^2 + \nu_z |k_3|^2) \widetilde{U}_{\bk} + f_0 \widetilde{U}_{\bk}^{\perp}\Big) dt =  \sigma_0|{\bk}|^{-\gamma} c_{\bk} dW_{\bk}, \text{ \, when\, } k_3\neq 0. \label{eq:tildeU}
\end{align}
Solving \eqref{eq:barU} gives
\begin{align}\label{eqn:Ukbar}
   &\overline{U}_{\bk}(t)  =   \overline{U}_{\bk}(0) e^{-\nu_h |\bk|^2 t} + \sigma_0|{\bk}|^{-\gamma} c_{\bk} \int_0^t e^{-\nu_h |\bk|^2 (t-s)} dW_{\bk}(s), \text{ \, when\, } k_3=0.
\end{align}
For $\widetilde{U}_{\bk}$, we rewrite equation \eqref{eq:tildeU} as
\begin{align*}
    d\widetilde{U}_{\bk} + M \widetilde{U}_{\bk}  dt = \sigma_0|{\bk}|^{-\gamma} c_{\bk} dW_{\bk},
\end{align*}
where
\[
M=\begin{pmatrix}
\nu_h |\bk'|^2 + \nu_z |k_3|^2  &  -f_0
\\
f_0 & \nu_h |\bk'|^2 + \nu_z |k_3|^2
\end{pmatrix},
\]
whose solution reads
\begin{equation}\label{eqn:Uktilde}
    \widetilde{U}_{\bk}(t) =  e^{-Mt}\widetilde{U}_{\bk}(0)  + \sigma_0|{\bk}|^{-\gamma} \int_0^t e^{-M(t-s)} c_{\bk} dW_{\bk}(s).
\end{equation}

We immediately obtain the following results about concerning the moments of the  Fourier modes $\overline U_{\bk}$ and $\widetilde U_{\bk}$.

\begin{lemma}\label{lem:U}
    Suppose that $U(0) = 0$. 
    Then, as $|\bk|\to \infty$, one has
\begin{align}
    &\mathbb E \int_0^T |\overline{U}_{\bk}|^2(t) dt \asymp \frac{\sigma_0^2 T|\bk|^{-2\gamma -2}}{2\nu_h}, \quad k_3=0, \label{order-Ukbar} \\
    &\mathbb E \int_0^T |\widetilde{U}_{\bk}|^2(t) dt \sim \frac{\sigma_0^2 T|\bk|^{-2\gamma}}{2(\nu_h |\bk'|^2 + \nu_z |k_3|^2)} , \quad k_3\neq 0, \label{order-Uktilde}
\end{align}
and 
\begin{equation}\label{order-Var}
    \Var\left[\int_0^T  |\overline{U}_{\bk}|^2(t) dt \right] \sim |\bk|^{-(4\gamma +6)},  \quad \Var\left[\int_0^T  |\widetilde{U}_{\bk}|^2(t) dt \right]  \sim \frac{|\bk|^{-4\gamma}}{(\nu_h |\bk'|^2 + \nu_z |k_3|^2)^3}.
\end{equation}
\end{lemma}

\begin{proof}
The results for $\overline U_{\bk}$ follow by direct computations and using It\^o's isometry. Regarding $\mathbb E \int_0^T |\widetilde{U}_{\bk}|^2(t)dt$, it suffices to compute
\begin{align*}
    \mathbb E |\widetilde{U}_{\bk}|^2(t)& =  \sigma_0^2|\bk|^{-2\gamma} \mathbb E \Big( \int_0^t e^{-M(t-s)} c_{\bk} dW_{\bk} (s)\Big)^2 
 = \sigma_0^2|\bk|^{-2\gamma} \mathbb E  \Big( \int_0^t c_{\bk}^{T} e^{-(M+M^T)(t-s)} c_{\bk}  ds\Big)
 \\
 &= \sigma_0^2|\bk|^{-2\gamma} \mathbb E  \Big( \int_0^t e^{-2(\nu_h |\bk'|^2 + \nu_z |k_3|^2)(t-s)}  ds\Big) = \sigma_0^2\frac{|\bk|^{-2\gamma}}{2(\nu_h |\bk'|^2 + \nu_z |k_3|^2)} \left(1- e^{-2(\nu_h |\bk'|^2 + \nu_z |k_3|^2)t}\right),
\end{align*}
where we have used the fact that $M$ and $M^T$ commute, implying $e^{-tM}e^{-tM^T} = e^{-t(M+M^T)}$. To establish \eqref{order-Var}, we compute 
\begin{equation}\label{eqn:UbarVar}
    \mathbb E\bigg[\Big(\int_0^T  |\overline{U}_{\bk}|^2(t) dt \Big)^2\bigg] = 
\int_0^T \int_0^T \mathbb E \left [|\overline{U}_{\bk}|^2(t)|\overline{U}_{\bk}|^2(s)\right]dt ds = 2\int_0^T \int_0^t \mathbb E \left [|\overline{U}_{\bk}|^2(t)|\overline{U}_{\bk}|^2(s)\right]ds dt,
\end{equation}
and for $s < t$, we derive
\begin{align*}
     \mathbb E \left [|\overline{U}_{\bk}|^2(t)|\overline{U}_{\bk}|^2(s)\right] & =\sigma_0^4|{\bk}|^{-4\gamma} e^{-2\nu_h |\bk|^2 (t+s)}
     \mathbb E \left[ \Big( \int_0^s e^{\nu_h |\bk|^2 u} dW_{\bk}(u)\Big)^4  \right] \\
     & \quad + \sigma_0^4|{\bk}|^{-4\gamma} e^{-2\nu_h |\bk|^2 (t+s)}\mathbb E \left[ \Big( \int_0^s e^{\nu_h |\bk|^2 u} dW_{\bk}(u)\Big)^2\Big( \int_s^t e^{\nu_h |\bk|^2 u} dW_{\bk}(u)\Big)^2 \right] \\
     & = \sigma_0^4|{\bk}|^{-4\gamma} e^{-2\nu_h |\bk|^2 (t+s)}\left[3\Big(\frac{e^{2\nu_h |\bk|^2 s}-1}{2\nu_h |\bk|^2 }\Big)^2 + \frac{e^{2\nu_h |\bk|^2 s}-1}{2\nu_h |\bk|^2 }  \frac{e^{2\nu_h |\bk|^2 t}-e^{2\nu_h |\bk|^2 s}}{2\nu_h |\bk|^2 } \right] \\
     & = \sigma_0^4 |{\bk}|^{-4\gamma} \Big(\frac{1}{2\nu_h |{\bk}|^2}\Big)^2 \left[1 - e^{-2\nu_h |\bk|^2 s} +2e^{-2\nu_h |\bk|^2 (t-s)} -5 e^{-2\nu_h |\bk|^2 t} + 3e^{-2\nu_h |\bk|^2 (t+s)}  \right].
\end{align*}
Plugging this back into equation~\eqref{eqn:UbarVar}, the integral of the first term in the bracket cancels out with $\mathbb E^2[\int_0^T  |\overline{U}_{\bk}|^2(t) dt]$, while the remaining terms are of order $\frac{1}{2\nu_h |\bk|^2}$. Combining these with the coefficients in front of the bracket, we conclude that the variance is of order $|\bk|^{-4\gamma - 6}$. The asymptotics for the variance involving $\widetilde{U}_{\bk}$ follows similarly. 
\end{proof}

\begin{remark}\label{rmk:nolimit}
    Multiplying  the right sides of \eqref{order-Ukbar} and \eqref{order-Uktilde} by $|\bk|^{2\gamma+2}$, we see that the right hand side of \eqref{order-Ukbar} becomes a constant $\frac{\sigma_0^2T}{2\nu_h}$ and thus converges, but the right hand side of \eqref{order-Uktilde} has no limit as $|\bk|\to \infty$. This is precisely the key reason why the estimator for $\nu_z$ derived from \eqref{PE-baroclinic} lacks the asymptotic normality property. To address this limitation, we introduce a new projection $\widehat V = \sum_{|\bk'|=\sqrt{q}|k_3|\neq 0} V_{\bk} \phi_{\bk}$, and in Section~\ref{sec:derivation}, we propose an estimator for $\nu_z$ based on $\widehat V$, which will be proved to be asymptotically normal; see Theorem~\ref{thm:normality}. This suggests that a careful selection of Fourier modes yields more accurate convergence results.
\end{remark}

Denote by $\widehat{U}_{\bk} = U_{\bk}$, when $|\bk'|=\sqrt q|k_3|$. Following similar computations as for $\widetilde U_{\bk}$, one can show that 
\begin{equation}
    \mathbb E \int_0^T |\widehat{U}_{\bk}|^2(t) dt \asymp \frac{\sigma_0^2 T|\bk|^{-2\gamma}}{2(\nu_h + \frac{\nu_z}{q} ) |\bk'|^2 } = \frac{\sigma_0^2 T|\bk|^{-2\gamma-2}}{\frac{2q}{q+1}\nu_h + \frac{2}{q+1}\nu_z} , \quad k_3\neq 0. \label{order-Ukhat}
\end{equation} 
Direct calculations yield
\begin{equation}\label{order-Vhat}
   \Var\left[\int_0^T  |\widehat{U}_{\bk}|^2(t) dt \right] \sim |\bk|^{-(4\gamma +6)}. 
\end{equation}
It is important to note that the right-hand side of \eqref{order-Ukhat} admits a limit $\frac{\sigma_0^2 T}{\frac{2q}{q+1}\nu_h + \frac{2}{q+1}\nu_z}$ after being multiplied by $|\bk|^{2\gamma+2}$. This is the primary motivation for our specific choice of $|\bk'|=\sqrt q|k_3|$ in defining the projection $\widehat{V}$. 

Note that since $|U_{\bk}|^2 = |\overline{U}_{\bk}|^2 + |\widetilde{U}_{\bk}|^2$, we have
\begin{equation}\label{order-Uk}
    \mathbb E \int_0^T |U_{\bk}|^2(t) dt \sim |\bk|^{-2\gamma -2}.
\end{equation}
Next, we define $U^N, \overline U^N, \widetilde U^N, \widehat U^N$ as follows 
\begin{equation}\label{eq:Un}
U^N = \sum\limits_{1\leq|\bk|\leq N} U_{\bk} \phi_{\bk},\quad \overline U^N = \sum\limits_{\substack{1\leq|\bk|\leq N\\ k_3=0}} U_{\bk} \phi_{\bk},\quad \widetilde U^N = \sum\limits_{\substack{1\leq|\bk|\leq N\\ k_3\neq 0}} U_{\bk} \phi_{\bk}, \quad \widehat U^N = \sum\limits_{\substack{1\leq|\bk|\leq N\\ |\bk'|=\sqrt q|k_3|}} U_{\bk} \phi_{\bk}.
\end{equation}

Thanks to \eqref{order-Ukbar}--\eqref{order-Ukhat} and Lemma \ref{lemma:order}, we have the following asymptotics regarding the linear parts.

\begin{lemma}\label{lemma:order-linear}
For $\beta>\frac\gamma2$, 
\begin{align}
     &\mathbb E \int_0^T \|A^\beta \overline{U}^N\|^2 dt \asymp \sigma_0^2\frac{T}{2\nu_h} \frac{\pi}{2\beta-\gamma} N^{4\beta-2\gamma}, \label{order-Ubar}
     \\
     & \mathbb E \int_0^T \|A^\beta \widetilde{U}^N\|^2 dt \sim  N^{4\beta-2\gamma+1}, \label{order-Utilde}
     \\
     &\mathbb E \int_0^T \|A^\beta U^N\|^2 dt \sim N^{4\beta-2\gamma+1}, \label{order-U}
     \\
     &\mathbb E \int_0^T \|A^\beta \widehat{U}^N\|^2 dt \asymp \sigma_0^2\frac{T}{\nu_h + \frac1q\nu_z}\frac{\pi}{2\beta-\gamma} N^{4\beta-2\gamma}, \label{order-Uhat}
\end{align}
and for $\beta_1+\beta_2+\beta_3 > \frac\gamma2$, one has
\begin{align}
 &\mathbb E \int_0^T \|A_h^{\beta_1} A_z^{\beta_2} A^{\beta_3} U^N\|^2 dt \sim \mathbb E \int_0^T \|A^{\beta_1+\beta_2+\beta_3} U^N\|^2 dt \sim N^{4(\beta_1+\beta_2+\beta_3)-2\gamma+1}, \label{eqn:same-order-1}
 \\
 &\mathbb E \int_0^T \|A_h^{\beta_1} A_z^{\beta_2} A^{\beta_3} \widetilde U^N\|^2 dt \sim \mathbb E \int_0^T \|A_h^{\beta_1} A_z^{\beta_2} A^{\beta_3} U^N\|^2 dt \sim N^{4(\beta_1+\beta_2+\beta_3)-2\gamma+1}, \label{eqn:same-order-2}
 \\
&\mathbb E \int_0^T \|A_h^{\beta_1} A_z^{\beta_2} A^{\beta_3} \widehat{U}^N\|^2 dt \sim \mathbb E \int_0^T \|A^{\beta_1+\beta_2+\beta_3} \widehat U^N\|^2 dt \sim N^{4(\beta_1+\beta_2+\beta_3)-2\gamma}. \label{eqn:same-order-3}
\end{align}
\end{lemma}
\begin{proof}
For $N\gg 1$, we direct evaluations, we obtain
\begin{align*}
    \mathbb E \int_0^T \|A^\beta \overline{U}^N\|^2 dt& = \mathbb E \int_0^T \left|\sum\limits_{ k_3=0, 1\leq |\bk|\leq N} |\bk|^{2\beta} U_{\bk} \phi_{\bk}\right|^2 dt = \sum\limits_{  k_3=0, 1\leq |\bk|\leq N} |\bk|^{4\beta} \mathbb E \int_0^T |\overline U_{\bk}|^2 dt \nonumber
    \\
    &\asymp \sigma_0^2\frac{T}{2\nu_h}\sum\limits_{  k_3=0,  1\leq |\bk|\leq N} |\bk|^{4\beta-2\gamma-2} = \sigma_0^2\frac{T}{2\nu_h} \sum\limits_{\bk'\in\mathbb Z^2, 1\leq |\bk'|\leq N} |\bk'|^{4\beta-2\gamma-2} \nonumber \\
    &\asymp \sigma_0^2\frac{T}{2\nu_h} \frac{\pi}{2\beta-\gamma} N^{4\beta-2\gamma},
\end{align*}
and
\begin{align*}
    \mathbb E \int_0^T \|A^\beta \widetilde{U}^N\|^2 dt& = \mathbb E \int_0^T \left|\sum\limits_{ k_3\neq 0, 1\leq |\bk|\leq N} |\bk|^{2\beta} U_{\bk} \phi_{\bk}\right|^2 dt = \sum\limits_{  k_3\neq 0, 1\leq |\bk|\leq N} |\bk|^{4\beta} \mathbb E \int_0^T |\widetilde U_{\bk}|^2 dt \nonumber
    \\
    &\sim \sum\limits_{ k_3\neq 0, 1\leq |\bk|\leq N} |\bk|^{4\beta-2\gamma-2} \sim \sum\limits_{\bk\in\mathbb Z^3, 1\leq |\bk|\leq N} |\bk|^{4\beta-2\gamma-2}  -  \sum\limits_{ k_3 = 0, 1\leq |\bk|\leq N} |\bk|^{4\beta-2\gamma-2}\nonumber
    \\
    &\sim  \sum\limits_{ \bk\in\mathbb Z^3, 1\leq |\bk|\leq N} |\bk|^{4\beta-2\gamma-2}  - \sum\limits_{\bk'\in\mathbb Z^2, 1\leq |\bk'|\leq N} |\bk'|^{4\beta-2\gamma-2}  \nonumber
    \\
    &\sim N^{4\beta-2\gamma+1} - N^{4\beta-2\gamma} \sim  N^{4\beta-2\gamma+1}.
\end{align*}
Likewise, we establish
\begin{align*}
     \mathbb E \int_0^T \|A^\beta \widehat{U}^N\|^2 dt& =  \mathbb E \int_0^T \left|\sum\limits_{ 1\leq |\bk|\leq N, |\bk'|=\sqrt{q}|k_3|} |\bk|^{2\beta} U_{\bk} \phi_{\bk}\right|^2 dt = \sum\limits_{ 1\leq |\bk|\leq N, |\bk'|=\sqrt{q}|k_3|} |\bk|^{4\beta} \mathbb E \int_0^T |\widehat U_{\bk}|^2 dt \nonumber
     \\
      &\asymp \sigma_0^2\frac{T}{\frac{2q}{q+1}\nu_h + \frac{2}{q+1}\nu_z}\sum\limits_{ 1\leq |\bk|\leq N, |\bk'|=\sqrt{q}|k_3|} |\bk|^{4\beta-2\gamma-2} 
      \\
      &= \sigma_0^2\frac{2T}{\frac{2q}{q+1}\nu_h + \frac{2}{q+1}\nu_z} \left(\frac{q+1}{q}\right)^{2\beta-\gamma-1} \sum\limits_{\bk'\in\mathbb Z^2, 1\leq |\bk'|\leq \sqrt{\frac{q}{q+1}}N}  |\bk'|^{4\beta-2\gamma-2} 
      \\
      &\asymp \sigma_0^2\frac{T}{\nu_h + \frac1q\nu_z} \frac{\pi}{2\beta-\gamma} N^{4\beta-2\gamma}.
\end{align*}
As $\|A^\beta U^N\|^2 = \|A^\beta \overline{U}^N\|^2 + \|A^\beta \widetilde{U}^N\|^2$, \eqref{order-U} follows immediately. 

Next notice that $$c_\alpha (|k_1|^\alpha + |k_2|^\alpha +|k_3|^\alpha) \leq |\bk|^{\alpha} \leq C_\alpha (|k_1|^\alpha + |k_2|^\alpha +|k_3|^\alpha),$$
and due to symmetry,
\begin{equation}\label{eqn:same-order}
    \sum\limits_{\bk\in\mathbb Z^d, 1\leq |\bk|\leq N} |k_1|^\alpha =  \sum\limits_{\bk\in\mathbb Z^d, 1\leq |\bk|\leq N} |k_2|^\alpha = \sum\limits_{\bk\in\mathbb Z^d, 1\leq |\bk|\leq N} |k_3|^\alpha \sim \sum\limits_{\bk\in\mathbb Z^d, 1\leq |\bk|\leq N} |\bk|^\alpha.
\end{equation}
From this we obtain \eqref{eqn:same-order-1}. 

For \eqref{eqn:same-order-2}, by direct calculation, we deduce 
\begin{align*}
    \mathbb E \int_0^T \|A_h^{\beta_1} A_z^{\beta_2} A^{\beta_3} \widetilde U^N\|^2 dt &= \mathbb E \int_0^T \|A_h^{\beta_1} A_z^{\beta_2} A^{\beta_3}  U^N\|^2 dt - \mathbb E \int_0^T \|A_h^{\beta_1} A_z^{\beta_2} A^{\beta_3}  \overline U^N\|^2 dt
    \\
    &=\mathbb E \int_0^T \|A_h^{\beta_1} A_z^{\beta_2} A^{\beta_3}  U^N\|^2 dt - \mathbb E \int_0^T \|A^{\beta_1+\beta_3} \overline U^N\|^2 dt
    \\
    &\sim N^{4(\beta_1+\beta_2+\beta_3)-2\gamma+1} - N^{4(\beta_1+\beta_3)-2\gamma} \sim N^{4(\beta_1+\beta_2+\beta_3)-2\gamma+1}.
\end{align*}

Finally, for \eqref{eqn:same-order-3}, notice that for $|\bk'|=\sqrt{q}|k_3|$  we have $|\bk'|^{4\beta_1} |k_3|^{4\beta_2} |\bk|^{4\beta_3} \sim |\bk|^{4(\beta_1+\beta_2+\beta_3)}$, and therefore,
\begin{align*}
    \mathbb E \int_0^T \|A_h^{\beta_1} A_z^{\beta_2} A^{\beta_3} \widehat U^N\|^2 dt \sim \mathbb E \int_0^T \|A^{\beta_1+\beta_2+\beta_3} \widehat U^N\|^2 dt \sim N^{4(\beta_1+\beta_2+\beta_3)-2\gamma}.
\end{align*}
This concludes the proof. 
\end{proof}

The next result provides the regularity of the solution to the linear system~\eqref{PE-original-linear}.

\begin{lemma}\label{lemma:regularity-linear}
Suppose that $U=\overline{U}+\widetilde U$ is a solution to \eqref{PE-original-linear}. Assume that $U(0) \in L^2(\Omega;\mathcal D(A^{\eta}))$ with $\eta\geq 0$, and $\gamma>\frac32$. Then for any $\gamma'<\frac\gamma2-\frac34$ and
$$\beta=
\begin{cases}
\eta, &\eta<\frac\gamma2-\frac34,
\\
\gamma', &\eta\geq \frac\gamma2-\frac34,
\end{cases}
$$
we have
\[
U,\; \widetilde U \in L^2(\Omega; L_{\text{loc}}^2((0,\infty);\mathcal D(A^{\beta+\frac12}))\cap C([0,\infty); \mathcal D(A^{\beta}) ) ).
\]
Also, for any $\gamma'<\frac\gamma2-\frac12$ and
$$\beta=
\begin{cases}
\eta, &\eta<\frac\gamma2-\frac12,
\\
\gamma', &\eta\geq \frac\gamma2-\frac12,
\end{cases}
$$
we have
\[
\overline{U}, \widehat U \in L^2(\Omega; L_{\text{loc}}^2((0,\infty);\mathcal D(A^{\beta+\frac12}))\cap C([0,\infty); \mathcal D(A^{\beta}) ) ).
\]
\end{lemma}
The proof follows from, e.g., \cite{da2014stochastic,rozovsky2018stochastic}. Specifically, the rotation part $U^\perp$ disappears in the energy estimate, and the estimate is essentially the same as that of the heat equation with additive noise. 

\begin{remark}\label{rmk:linear-improve}
    The difference in the regularities of $U, \widetilde U$ and $\overline{U}, \widehat U$ arises from the dimensionality difference:  $U, \widetilde U$ are three-dimensional, while $\overline{U}, \widehat U$ are essentially two-dimensional. If one considers the horizontal average of $U$, it results in a one-dimensional function, further improving  the regularity by $\frac14$.
\end{remark}

\section{Derivation of the Estimators and the Main Results}\label{sec:derivation}

In this section, we outline the heuristic derivations of the estimators based on the Girsanov theorem. The procedure and reasoning follow to \cite{cialenco2011parameter}.
We then state the main results of this paper. The estimators are based on the first $N$ Fourier modes of the solution $V$ to the original system \eqref{PE-system}. To this end, similar to \eqref{eq:Un},  we define $V_{\bk} = \langle V, \phi_{\bk}\rangle$ and the various projections $V^N, \overline V^N, \widetilde V^N, \widehat V^N$ of  $V$ as follows:
\[
V^N = \sum\limits_{1\leq|\bk|\leq N} V_{\bk} \phi_{\bk} ,\quad \overline V^N  = \sum\limits_{\substack{1\leq|\bk|\leq N\\ k_3=0}} V_{\bk} \phi_{\bk},\quad \widetilde V^N  = \sum\limits_{\substack{1\leq|\bk|\leq N\\ k_3\neq 0}} V_{\bk} \phi_{\bk}, \quad \widehat V^N = \sum\limits_{\substack{1\leq|\bk|\leq N\\ |\bk'|=\sqrt{q}|k_3|}} V_{\bk} \phi_{\bk}.
\]
Denote by $B_N(V,V) = P_N B(V,V)$, then $\overline{V}^N$, $\widetilde{V}^N$, and $\widehat{V}^N$ satisfy
\begin{subequations}\label{PE-original-finitemode}
\begin{align}
    &d \overline{V}^N + \left(\mathcal P_h \overline{B_N(V,V)} + \nu_h A_h \overline{V}^N\right)dt
    = \sigma_0\sum\limits_{k_3=0, 1\leq|\bk|\leq N} |{\bk}|^{-\gamma} c_{\bk}\phi_{\bk} dW_{\bk},
    \\
    &d \widetilde V^N + \left(\widetilde{B_N(V,V)} + \nu_h A_h \widetilde V^N + \nu_z A_z \widetilde V^N + f_0 \widetilde{(V^N)}^\perp\right)dt
     = \sigma_0\sum\limits_{k_3\neq 0, 1\leq|\bk|\leq N} |{\bk}|^{-\gamma} c_{\bk}\phi_{\bk} dW_{\bk},
     \\
     &d \widehat V^N + \left(\widehat{B_N(V,V)} + \nu_h  A_h \widehat V^N + \nu_z A_z \widehat V^N + f_0 \widehat{(V^N)}^\perp\right)dt
     = \sigma_0\sum\limits_{ 1\leq|\bk|\leq N, |\bk'|=\sqrt{q}|k_3|} |{\bk}|^{-\gamma} c_{\bk}\phi_{\bk} dW_{\bk}.
\end{align}
\end{subequations}
Let  $\kappa=(\nu_h,\nu_z)$ and $\kappa_0 = (\nu_{h0},\nu_{z0})$, and  $\mathbb P_{\kappa}^{\overline{V},N,T}$ and $\mathbb P_{\kappa}^{\widetilde{V},N,T}$ be the probability measures in $C([0,T];\mathbb R^N)$ generated by $\overline{V}^N$ and $\widetilde{V}^N$, respectively, under the parametrization $\kappa=(\nu_h,\nu_z)$. Using \cite[Section~7.6.4]{liptser2001statistics}, we can informally write the Radon–Nikodym derivatives $ \frac{d\mathbb P_{\kappa}^{\overline{V},N,T}(\overline{V}^N)}{d\mathbb P_{\kappa_0}^{\overline{V},N,T}}$ and $\frac{d\mathbb P_{\kappa}^{\widetilde{V},N,T}(\widetilde{V}^N)}{d\mathbb P_{\kappa_0}^{\widetilde{V},N,T}}$  as follows:
\begin{align*}
    \frac{d\mathbb P_{\kappa}^{\overline{V},N,T}(\overline{V}^N)}{d\mathbb P_{\kappa_0}^{\overline{V},N,T}} = & \exp\Big[- \frac1{\sigma_0^2} \Big(\int_0^T  (\nu_h - \nu_{h0}) \langle A_h^{1+\gamma} \overline{V}^N, d\overline{V}^N \rangle
    \\
    & \quad - \frac12 \int_0^T (\nu_h^2 - \nu_{h0}^2) \langle A_h^{1+\frac{\gamma}{2}} \overline{V}^N,  A_h^{1+\frac{\gamma}{2}} \overline{V}^N \rangle dt - \int_0^T  (\nu_h - \nu_{h0}) \langle A_h^{1+\gamma} \overline{V}^N , \mathcal P_h\overline{B_N(V,V)} \rangle dt \Big)\Big],
\end{align*}
\begin{align*}
    \frac{d\mathbb P_{\kappa}^{\widetilde{V},N,T}(\widetilde{V}^N)}{d\mathbb P_{\kappa_0}^{\widetilde{V},N,T}} = & \exp\Big[- \frac1{\sigma_0^2} \Big(\int_0^T  (\nu_h - \nu_{h0}) \langle A_h A^{\gamma} \widetilde{V}^N, d\widetilde{V}^N \rangle - \int_0^T  (\nu_z - \nu_{z0}) \langle A_z A^{\gamma} \widetilde{V}^N, d\widetilde{V}^N \rangle
    \\
    & \quad - \frac12 \int_0^T (\nu_h^2 - \nu_{h0}^2) \langle A_h A^{\gamma} \widetilde{V}^N,  A_h \widetilde{V}^N \rangle dt - \frac12 \int_0^T (\nu_z^2 - \nu_{z0}^2) \langle A_z A^{\gamma} \widetilde{V}^N,  A_z \widetilde{V}^N \rangle dt
    \\
    & \quad -\int_0^T (\nu_h \nu_z - \nu_{h0}\nu_{z0}) \langle A_h A^{\gamma} \widetilde{V}^N,  A_z  \widetilde{V}^N \rangle dt
    \\
    & \quad - \int_0^T  (\nu_h - \nu_{h0}) \left\langle A_h A^{\gamma} \widetilde{V}^N , \widetilde{B_N(V,V)}  \right\rangle dt 
   - \int_0^T  (\nu_z - \nu_{z0}) \left\langle A_z A^{\gamma} \widetilde{V}^N , \widetilde{B_N(V,V)}  \right\rangle dt \Big)\Big] .
\end{align*}
By maximizing the likelihood ratio $d\mathbb P_{\kappa}^{\overline{V},N,T} /d\mathbb P_{\kappa_0}^{\overline{V},N,T}$ with respect to $\nu_h$, we may compute its informal maximum likelihood estimator (MLE): 
\begin{align*}
   \nu^N_{h1,MLE} = -\frac{\int_0^T  \langle A_h^{1+\gamma} \overline{V}^N, d\overline{V}^N \rangle + \int_0^T  \langle A_h^{1+\gamma} \overline{V}^N , \mathcal P_h\overline{B_N(V,V)} \rangle dt}{\int_0^T \|A_h^{1+\frac{\gamma}{2}} \overline{V}^N\|^2 dt}.
\end{align*}
Following its formalization, we propose the following class of estimators: 
\begin{align*}
   \nu^N_{h1} := -\frac{\int_0^T  \langle A_h^{1+\alpha} \overline{V}^N, d\overline{V}^N \rangle + \int_0^T  \langle A_h^{1+\alpha} \overline{V}^N , \mathcal P_h\overline{B_N(V,V)} \rangle dt}{\int_0^T \|A_h^{1+\frac{\alpha}{2}} \overline{V}^N\|^2 dt},
\end{align*}
where $\alpha$ is a free parameter whose range will be specified later. Note that $\nu^N_{h1,MLE}$ is a particular case of $\nu_{h1}^N$ when $\alpha=\gamma$. Next, by maximizing $ d\mathbb P_{\kappa}^{\widetilde{V},N,T}/d\mathbb P_{\kappa_0}^{\widetilde{V},N,T}$, and replacing $\nu_h$ by $\nu^N_{h1}$, we achieve the following candidate estimator for $\nu_z$:
\begin{align}
   \nu^N_{z1} := & -\frac1{\int_0^T \|A_z A^{\frac{\alpha}{2}} \widetilde{V}^N\|^2 dt}\Bigg(\int_0^T  \langle A_z A^{\alpha} \widetilde{V}^N, d\widetilde{V}^N \rangle  + \nu^N_{h1} \int_0^T \langle A_h A^{\alpha} \widetilde{V}^N,  A_z  \widetilde{V}^N \rangle dt \nonumber
   \\
   &\hspace{4cm}+ \int_0^T \left\langle A_z A^{\alpha} \widetilde{V}^N , \widetilde{B_N(V,V)}  \right\rangle dt\Bigg)\nonumber
   \\
   =& -\frac1{\int_0^T \|A_z A^{\frac{\alpha}{2}} \widetilde{V}^N\|^2 dt}\Bigg(\int_0^T  \langle A_z A^{\alpha} \widetilde{V}^N, d\widetilde{V}^N \rangle +     \int_0^T \left\langle A_z A^{\alpha} \widetilde{V}^N , \widetilde{B_N(V,V)} \right\rangle dt
\label{eq:nuz1}   \\
   &\hspace{2cm}-\frac{\int_0^T  \langle A_h^{1+\alpha} \overline{V}^N, d\overline{V}^N \rangle + \int_0^T  \langle A_h^{1+\alpha} \overline{V}^N , \mathcal P_h\overline{B_N(V,V)} \rangle dt}{\int_0^T \|A_h^{1+\frac{\alpha}{2}} \overline{V}^N\|^2 dt}\int_0^T \langle A_h A^{\alpha} \widetilde{V}^N,  A_z  \widetilde{V}^N \rangle dt\Bigg). \nonumber
\end{align}
While $\nu_{y1}^N,\nu_{z1}^N$ are valid estimators, we remark that both depend on $B_N(V, V)$, hence require knowledge of the entire path of the solution on time interval $[0,T]$, not only of the first $N$ Fourier modes as assumed by our sampling scheme. To overcome this drawback, we introduce two additional classes of estimators, where we replace $B_N(V, V)$ by its Galerkin type projection $B_N(V^N,V^N)$. This leads to 
\begin{align*}
   \nu^N_{h2} := -\frac{\int_0^T  \langle A_h^{1+\alpha} \overline{V}^N, d\overline{V}^N \rangle + \int_0^T  \langle A_h^{1+\alpha} \overline{V}^N , \mathcal P_h\overline{B_N(V^N,V^N)} \rangle dt}{\int_0^T \|A_h^{1+\frac{\alpha}{2}} \overline{V}^N\|^2 dt},
\end{align*}
and
\begin{align*}
   \nu^N_{z2} :=& -\frac1{\int_0^T \|A_z A^{\frac{\alpha}{2}} \widetilde{V}^N\|^2 dt}\Bigg(\int_0^T  \langle A_z A^{\alpha} \widetilde{V}^N, d\widetilde{V}^N \rangle +     \int_0^T \left\langle A_z A^{\alpha} \widetilde{V}^N , \widetilde{B_N(V^N,V^N)} \right\rangle dt
   \\
   &\hspace{2cm}-\frac{\int_0^T  \langle A_h^{1+\alpha} \overline{V}^N, d\overline{V}^N \rangle + \int_0^T  \langle A_h^{1+\alpha} \overline{V}^N , \mathcal P_h\overline{B_N(V^N,V^N)} \rangle dt}{\int_0^T \|A_h^{1+\frac{\alpha}{2}} \overline{V}^N\|^2 dt}\int_0^T \langle A_h A^{\alpha} \widetilde{V}^N,  A_z  \widetilde{V}^N \rangle dt\Bigg).
\end{align*}
As one may expect, these estimators are `not far' from  $\nu_{y1}^N,\nu_{z1}^N$, and as we show below, they indeed remain consistent; see Theorem~\ref{thm:consistency}. Moreover, in the above estimators, one can drop all together the terms that involve the nonlinear component $B$ and consider the following estimators
\begin{align*}
   \nu^N_{h3} := -\frac{\int_0^T  \langle A_h^{1+\alpha} \overline{V}^N, d\overline{V}^N \rangle dt }{\int_0^T \|A_h^{1+\frac{\alpha}{2}} \overline{V}^N\|^2 dt},
\end{align*}
and
\begin{align*}
   \nu^N_{z3} := -\frac1{\int_0^T \|A_z A^{\frac{\alpha}{2}} \widetilde{V}^N\|^2 dt}\Bigg(\int_0^T  \langle A_z A^{\alpha} \widetilde{V}^N, d\widetilde{V}^N \rangle  dt
   -\frac{\int_0^T  \langle A_h^{1+\alpha} \overline{V}^N, d\overline{V}^N \rangle dt}{\int_0^T \|A_h^{1+\frac{\alpha}{2}} \overline{V}^N\|^2 dt}\int_0^T \langle A_h A^{\alpha} \widetilde{V}^N,  A_z  \widetilde{V}^N \rangle dt\Bigg).
\end{align*}
These estimators have clear computational advantages. Later on, we will show rigorously in Lemma \ref{lemma:nonlinear-term} that the neglected nonlinear terms are of lower orders and converge to zero as $N \to \infty$, which validates the choice of $\nu_{h3}^N$ and $\nu_{z3}^N$. 

As the next result shows, under appropriate conditions, all aforementioned estimators are (weakly) consistent.  Furthermore, we also prove that $\nu_{h1}^N$ is asymptotically normal with rate $N^2$. However, such property and the corresponding proofs do not extend to $\nu_{z1}^N$; see Remark~\ref{rmk:nolimit} for a detailed discussion. To address this, we propose a new class of estimators for $\nu_z$, by introducing a specially chosen projection. Namely, we replace all \texttt{tilde} projections $\sum_{{1\leq|\bk|\leq N, k_3\neq 0}}$ in \eqref{eq:nuz1} by \texttt{hat} projections $\sum_{{1\leq|\bk|\leq N, |\bk'| = \sqrt{q}|k_3|}}$, and consider 
\begin{align*}
   \widehat{\nu^N_{z1}}
   :=& -\frac1{\int_0^T \|A_z A^{\frac{\alpha}{2}} \widehat{V}^N\|^2 dt}\Bigg(\int_0^T  \langle A_z A^{\alpha} \widehat{V}^N, d\widehat{V}^N \rangle +     \int_0^T \left\langle A_z A^{\alpha} \widehat{V}^N , \widehat{B_N(V,V)} \right\rangle dt
   \\
   &\hspace{2cm}-\frac{\int_0^T  \langle A_h^{1+\alpha} \overline{V}^N, d\overline{V}^N \rangle + \int_0^T  \langle A_h^{1+\alpha} \overline{V}^N , \mathcal P_h\overline{B_N(V,V)} \rangle dt}{\int_0^T \|A_h^{1+\frac{\alpha}{2}} \overline{V}^N\|^2 dt}\int_0^T \langle A_h A^{\alpha} \widehat{V}^N,  A_z  \widehat{V}^N \rangle dt\Bigg),
\end{align*}
Analogously, one defines $\widehat{\nu_{z2}^N}$ and $\widehat{\nu_{z3}^N}$  using $\nu_{z2}^N$ and $\nu_{z3}^N$, respectively.

Now, we are ready to present the main results of our work.

\begin{theorem}[Consistency]\label{thm:consistency}
Assume that $\gamma>4$ and suppose that $V$ is the solution to the system \eqref{PE-system} with an initial condition $V(0) = V_0 \in \mathcal D(A^{\frac12+\gamma'})$ for all $\max\left\{\frac54, \frac\gamma2-1\right\}<\gamma'<\frac\gamma2-\frac34$. Then, for any $\alpha>\gamma-2$,

    (i)   $\nu^N_{h1}$, $\nu^N_{z1}$, and $\widehat{\nu^N_{z1}}$ are weakly consistent estimators of  $\nu_h$, $\nu_z$, and $\nu_z$, respectively, i.e.,
    \[
    \lim\limits_{N\to\infty} \nu^N_{h1} = \nu_h, \quad \lim\limits_{N\to\infty} \nu^N_{z1} = \lim\limits_{N\to\infty} \widehat{\nu^N_{z1}} = \nu_z,
    \]
    in probability; 

    (ii) if furthermore $\gamma>\frac92$, the estimators $\nu^N_{h2}$ and $\nu^N_{h3}$ are weakly consistent estimators of $\nu_h$, and $\nu^N_{z2}$ , $\widehat{\nu_{z2}^N}$, $\nu^N_{z3}$, $\widehat{\nu_{z3}^N}$ are weakly consistent estimators of  $\nu_z$.
\end{theorem}
The proof of Theorem~\ref{thm:consistency} is presented in Section~\ref{sec:consistency}. 

\begin{theorem}[Asymptotic normality]\label{thm:normality}
    Assume that $\gamma>4$, and suppose that $V$ is the solution to the system \eqref{PE-system} with an initial condition $V(0) = V_0 \in \mathcal D(A^{\frac12+\gamma'})$ for all $\max\left\{\frac54, \frac\gamma2-1\right\}<\gamma'<\frac\gamma2-\frac34$. Then for any $\alpha > \gamma-1$ and positive rational $q$,
    $\nu^N_{h1}$ and $\widehat{\nu_{z1}^N}$ are jointly asymptotically normal with rate $N^2$: 
\begin{equation*}
 N^2\left(
    \begin{array}{c}
    \nu_{h1}^N - \nu_h \\
    \widehat{\nu^N_{z1}}-\nu_z
    \end{array}
    \right)
    \stackrel{\mathcal{D}}\longrightarrow \mathcal{N}\left(
    \left[
    \begin{array}{c}
    0\\0
    \end{array}
    \right],
    \left[
    \begin{array}{cc}
    \frac{2\nu_h}{\pi T} \frac{(2+\alpha-\gamma)^2}{2+2\alpha-2\gamma} & - \frac{2q\nu_h}{\pi T} \frac{(2+\alpha-\gamma)^2}{2+2\alpha-2\gamma}\\
    -\frac{2q\nu_h}{\pi T} \frac{(2+\alpha-\gamma)^2}{2+2\alpha-2\gamma} & \frac{(2q^2 +q+1)\nu_h + (1+\frac1q)\nu_z}{\pi T} \frac{(2+\alpha-\gamma)^2}{2+2\alpha-2\gamma}
    \end{array}
    \right]
    \right),
\end{equation*}
where $\mathcal{N}(\mu, \Sigma)$ represents a multivariate normal random variable with mean vector $\mu$ and covariance matrix $\Sigma$.
\end{theorem}
The proof of Theorem~\ref{thm:normality} is deferred to Section~\ref{sec:normality}.

From Theorem~\ref{thm:normality}, it follows that the estimators $\nu_{h1}^N$ and $q\nu_{h1}^N + \widehat{\nu_{z1}^N}$ are asymptotically uncorrelated.  This will become clear once we rewrite $\nu_{h1}^N$ and $\widehat{\nu_{z1}^N}$ in equations \eqref{nuh1N} and \eqref{hatnuz1N}, and use the fact that the linear components of $\overline V^N$ and $\widehat V^N$ are independent.

\begin{remark}
    From the covariance matrix $\Sigma$ in Theorem 4.2, we note that the positive constant $q\in\mathbb Q$ that minimizes $f(q;\nu_h,\nu_z)=(2q^2 +q+1)\nu_h + (1+\frac1q)\nu_z$ is the optimal choice for minimum variance of $\widehat{\nu_{z1}^N}$. However, such a choice depends on the parameters of interest $\nu_h$ and $\nu_z$.
\end{remark}

\section{Regularity of the Solution}\label{sec:regularity}

This section focuses on the regularity analysis of $V$ defined in \eqref{PE-system}, as well as the residual $R:= V - U$ where $U$ satisfies \eqref{PE-original-linear}. These analytical properties are of independent theoretical interest, in addition to being fundamentally used in the proofs of the main theorems. The decomposition $V = U + R$ splits the solution $V$ to the system~\eqref{PE-system} into a linear part $U$, which satisfies a linear stochastic PDE with the initial condition $U(0)=0$, and a nonlinear part $R$, which satisfies the following nonlinear random PDE:
\begin{align}\label{eqn:R}
    dR + (\nu_h A_h R + \nu_z A_z R + f_0 \mathcal P_h R^{\perp}) dt + B(R+U,R+U) dt =0,
\end{align}
with initial data $R(0) = V_0$. 

We start with a result on the global well-posedness of the $3D$ viscous PE. 
\begin{lemma}\label{lemma:regularity-V}
  Assume that $V_0\in \mathcal D(A^{\eta})$ a.s. with $\eta\geq 1$ and $\gamma>\frac72$. Then, for any $\gamma'<\frac\gamma2-\frac34$ and
  \[
  \beta=
\begin{cases}
\eta, &\eta<\frac\gamma2-\frac34,
\\
\gamma', &\eta\geq \frac\gamma2-\frac34,
\end{cases}
  \]
  there exists a unique, $H$-valued, $\mathcal F_t$-adapted process $V$ such that
  \begin{equation}\label{regularity:V-1}
        V\in L_{\text{loc}}^2((0,\infty);\mathcal D(A^{\beta+\frac12})\cap C([0,\infty); \mathcal D(A^{\beta})) \quad a.s.,
  \end{equation}
  and so that for each $t\geq 0$,
  \[
   V(t) + \int_0^t \left( V\cdot \nabla_h V + w(V)\pp_z V -\nu_h \Delta_h V - \nu_z \pp_{zz}V + f_0 V^\perp\right) d\tilde t = V_0 + \sigma W(t) \text{ in } H.
  \]
\end{lemma}
\begin{proof}
%We first remark that from \cite{guo20093d} the above result holds when $\beta=\eta=\frac12$.
Note that $\sigma\in L_2(H, \mathcal D(A^{\frac\gamma2-\frac34-\varepsilon}))$ for sufficiently small $\varepsilon>0$. Since $\gamma>\frac72$, we specifically have $\sigma\in L_2(H, \mathcal D(A^{1 + \tilde\varepsilon}))$ for some small $\tilde{\varepsilon}>0$. From the definition of $\beta$, it is clear that when $\eta=1$ then $\beta=1$, and when $\eta>1$,  it follows at once that  $\beta> 1$ a large enough $\gamma'$. First we focus on the case $\beta=\eta=1$, and then proceed to $\beta>1$. For brevity, some estimates below are derived rather informally, but these estimates can be rigorously justified through the traditional  Galerkin approximation arguments: first obtain the estimates for the corresponding Galerkin scheme and then pass to the limit.

Assume that $U$ solves the linear system \eqref{PE-original-linear} with initial data $U(0)=0$. Then by Lemma~\ref{lemma:regularity-linear} we establish the regularity of $U$: 
\begin{equation}\label{regularity:U-1}
    U \in L_{\text{loc}}^2((0,\infty);\mathcal D(A^{\beta+\frac12}))\cap C([0,\infty); \mathcal D(A^{\beta})) \quad a.s..
\end{equation}
Now consider $R=V-U$ with $R(0) = V_0$. 
We first recall from \cite{guo20093d} that equation~\eqref{regularity:V-1} holds with $\beta=\frac12$, except that  $C([0,\infty); \mathcal D(A^{\frac12}))$ is replaced by $L_{\text{loc}}^\infty((0,\infty); \mathcal D(A^{\frac12}))$. We claim that indeed one has $C([0,\infty); \mathcal D(A^{\frac12}))$ as well. The result from \cite{guo20093d}, together with \eqref{regularity:U-1}, implies that
\begin{equation}\label{regularity:R-1}
     R \in L_{\text{loc}}^2((0,\infty);\mathcal D(A))\cap L_{\text{loc}}^\infty((0,\infty); \mathcal D(A^{\frac12})) \quad a.s.. 
\end{equation}
Once we prove $\frac{dR}{dt}\in L_{\text{loc}}^2((0,\infty);H)$, the regularity $R \in C([0,\infty); \mathcal D(A^{\frac12}))$ follows from the Lions–Magenes lemma. To this end, we estimate $\frac{dR}{dt}$ using equation~\eqref{eqn:R}. All the linear terms are readily addressed, so we focus on the nonlinear terms only. By Lemma~\ref{lemma:a2}, we have
\begin{align*}
    \|B(U+R,U+R)\| \leq C(\|A^{\frac12}U\| + \|A^{\frac12} R\|)(\|AU\| + \|AR\|).
\end{align*}
Using the regularity of $U$ and $R$ from equation~\eqref{regularity:U-1} and \eqref{regularity:R-1} , respectively, we conclude that $B(U+R,U+R)\in L_{\text{loc}}^2((0,\infty);H)$, which implies that $\frac{dR}{dt}\in L_{\text{loc}}^2((0,\infty);H)$, as required.

For $\beta=\eta=1$, multiplying \eqref{eqn:R} by $A^{2} R$ and integrating on $\mathbb T^3$ gives 
\begin{align}\label{R-1}
    \frac d{dt} \|A R\|^2 + 2\nu_h \|A_h^{\frac12} AR\|^2  + 2\nu_z \|A_z^{\frac12} A R\|^2  - 2 \langle A^{\frac12} B(R+U,R+U), A^{\frac32} R\rangle =0.
\end{align}
Applying Lemma~\ref{lemma:a2} and noticing that $\nabla_h\cdot \overline{R+U} = \nabla_h\cdot \overline V=0$, as the norms $\|A^{\frac n2} f\|\sim \|f\|_{\dot{H}^n}$ are equivalent for $n\in\mathbb N$, we get 
\begin{align*}
    \|A^{\frac12} B(R+U,&R+U)\| \leq C \|B(R+U,R+U)\|_{\dot{H}^1} 
    \\
    \leq &  C\left(\|B(\nabla(R+U),R+U)\| + \|B(R+U,\nabla(R+U))\| \right)
    \\
    \leq &C\left(\|A R\|^{\frac12} \|A^{\frac32} R\|^{\frac12} +  \|A U\|^{\frac12} \|A^{\frac32} U\|^{\frac12}\right)
    \left( \|A^{\frac12} R\|^{\frac12} \|A R\|^{\frac12}  +  \|A^{\frac12} U\|^{\frac12} \|A U\|^{\frac12}  \right).
\end{align*}
By Young's inequality, we estimate the nonlinear terms in \eqref{R-1} as
\begin{align*}
    |2\langle A^{\frac12}  B(R+U,&R+U), A^{\frac32} R\rangle| \leq C \|A^{\frac12} B(R+U,R+U)\| \|A^{\frac32} R\| 
    \\
    \leq &C \left(\|A R\|^{\frac12} \|A^{\frac32} R\|^{\frac12} +  \|A U\|^{\frac12} \|A^{\frac32} U\|^{\frac12}\right)
    \\
    &\times\left( \|A^{\frac12} R\|^{\frac12} \|A R\|^{\frac12}  +  \|A^{\frac12} U\|^{\frac12} \|A U\|^{\frac12}  \right) \|A^{\frac32} R\|
    \\
    \leq & \nu_h \|A_h^{\frac12} A R\|^2  + \nu_z \|A_z^{\frac12} A R\|^2 
    \\
    &+ C\Big( \|A^{\frac12} R\|^{2} \|A R\|^{2}   + \|A^{\frac12} U\|^{2} \|A U\|^{2} + \|A U\|^{2} \|A^{\frac32} U\|^2 +1 \Big) \Big(\|A R\|^2  + 1\Big).
\end{align*}
Therefore, we obtain
\begin{align}\label{ine:R}
    &\frac d{dt} \|A R\|^2 + \nu_h \|A_h^{\frac12} A R\|^2  + \nu_z \|A_z^{\frac12} A R\|^2 \nonumber
    \\
    \leq & C\Big( \|A^{\frac12} R\|^{2} \|A R\|^{2}   + \|A^{\frac12} U\|^{2} \|A U\|^{2} + \|A U\|^{2} \|A^{\frac32} U\|^2 +1 \Big) \Big(\|A R\|^2  + 1\Big). 
\end{align}
Thanks to \eqref{regularity:U-1} and \eqref{regularity:R-1}, as $R(0)=V_0\in \mathcal D(A)$, using the Gronwall inequality we infer that 
\begin{equation*}
    R \in L_{\text{loc}}^2((0,\infty);\mathcal D(A^{\frac32}))\cap L^\infty_{\text{loc}}((0,\infty); \mathcal D(A)) \quad a.s..
\end{equation*}
Thanks to the nonlinear estimate for $A^\frac12 B(U+R,U+R)$ above, we obtain $\frac{dR}{dt}\in L_{\text{loc}}^2((0,\infty);\mathcal D(A^{\frac12}))$. Applying the Lions–Magenes lemma once again, we conclude that
\begin{equation}\label{regularity:R-2}
    R \in L_{\text{loc}}^2((0,\infty);\mathcal D(A^{\frac32}))\cap C([0,\infty); \mathcal D(A)) \quad a.s..
\end{equation}
Combining \eqref{regularity:U-1} and \eqref{regularity:R-2}, we arrive at the desired \eqref{regularity:V-1}, assuming $\beta=\eta=1$.

% Next, we consider when $\beta\in  (1,\frac54]$ and $\eta\geq \beta >1$. The nonlinear term in \eqref{R-1} becomes 
% \begin{align*}
%     \langle A^{\beta} B(R+U,R+U), A^{\beta} R\rangle = \langle A^{\beta-\frac14} B(R+U,R+U), A^{\beta+\frac14} R\rangle.
% \end{align*}
% Thanks to Lemma \ref{lemma:a3}, as $\beta-\frac14>\frac34$, we have
% \begin{align*}
%     |\langle A^{\beta-\frac14} B(R+U,R+U), A^{\beta+\frac14} R\rangle| \leq C (\|A^{\beta+\frac14} R\|^2 + \|A^{\beta+\frac14} U\|^2)\|A^{\beta+\frac14} R\|.
% \end{align*}
% By the interpolation inequality, we know $\|A^{\beta+\frac14} f\| \leq C \|A^{\beta} f\|^{\frac12} \|A^{\beta+\frac12} f\|^{\frac12}$. Therefore, since $\beta+\frac14\leq \frac32$
% \begin{align*}
%     &|\langle A^{\beta-\frac14} B(R+U,R+U), A^{\beta+\frac14} R\rangle| 
%     \\
%     \leq & C (\|A^{\beta} R\| \|A^{\beta+\frac14} R\| + \|A^{\beta} U\| \|A^{\beta+\frac12} U\|)\|A^{\beta+\frac12} R\|
%     \\
%     \leq &  \nu_h \|A_h^{\frac12} A^{\beta} R\|^2  + \nu_z \|A_z^{\frac12} A^{\beta} R\|^2 + C (\|A^{\frac32} R\|^2 A^{\beta} R\|^2 + \|A^{\beta} U\|^2 \|A^{\beta+\frac12} U\|^2).
% \end{align*}
% The application of the Gronwall inequality and the fact that $R(0)=V_0\in \mathcal D(A^\eta)\subseteq \mathcal D(A^\beta)$ yield that
% \begin{equation}\label{regularity:R-5}
%     R \in L_{\text{loc}}^2((0,\infty);\mathcal D(A^{\beta+\frac12}))\cap C([0,\infty); \mathcal D(A^\beta)) \quad a.s..
% \end{equation}
% Combining \eqref{regularity:U-1} and \eqref{regularity:R-5}, we arrive at the desired \eqref{regularity:V-1} when $\beta\in (1,\frac54]$.

Next, we consider the case $\beta >1$ and $\eta\geq \beta>1$. 
%In this case, we will use the following standard fractional calculus for product (see, for example, \cite[Lemma A.1]{constantin2015long}):
% \begin{align}\label{ine:fractional}
%     \|A^r (fg) \| \leq C (\|A^r f\|\|g\|_{L^\infty} + \|f\|_{L^\infty}\|A^r g\|), \quad \forall r\geq 0.
% \end{align}
Using Lemma \ref{lemma:a1}, by virtue of the nonlinear structure of $B(f,g)$ and since $R+U=V$ has zero mean, we estimate the nonlinear term as
\begin{align*}
    |\langle A^{\beta} B(R+U,&R+U), A^{\beta} R\rangle| = |\langle A^{\beta-\frac12} B(R+U,R+U), A^{\beta+\frac12} R\rangle|
    \\
    \leq & C \|A^{\beta}(R+U)\| \|A^{\frac12+\frac34+\varepsilon}(R+U)\| \|A^{\beta+\frac12} R\|
    \\
    \leq &C \|A^{\frac32} V\| (\|A^\beta R\| + \|A^\beta U\|) \|A^{\beta+\frac12} R\|
    \\
    \leq &  \nu_h \|A_h^{\frac12} A^{\beta} R\|^2  + \nu_z \|A_z^{\frac12} A^{\beta} R\|^2 + C \|A^{\frac32} V\|^2 (\|A^\beta R\|^2 + \|A^\beta U\|^2),
\end{align*}
where we take $\varepsilon\leq \frac14$.
Thanks to previous step, we know that $V\in L_{\text{loc}}^2((0,\infty);\mathcal D(A^{\frac32})$. In view of \eqref{regularity:U-1} and combined with Gronwall inequality, we get 
\begin{equation}\label{regularity:R-5}
     R \in L_{\text{loc}}^2((0,\infty);\mathcal D(A^{\beta+\frac12}))\cap C([0,\infty); \mathcal D(A^\beta)) \quad a.s.,
 \end{equation}
 where the continuity in time follows from the Lions–Magenes lemma.
 Combining \eqref{regularity:U-1} and \eqref{regularity:R-5}, we arrive at the desired \eqref{regularity:V-1} when $\beta>1$. 
 The proof is complete. 

\end{proof}

\begin{remark}
    In contrast to Lemma~\ref{lemma:regularity-linear}, where we obtained `better regularity' for the linear parts $\overline{U}$ and $\widehat{U}$ due to dimension reduction, we do not expect similar results to hold for $\overline{V}$ and $\widehat{V}$. The reason is that the nonlinear terms will drag the impact of $\widetilde V$ and $V-\widehat{V}$, which are three-dimensional, into the evolution of $\overline{V}$ and $\widehat{V}$, respectively. 
\end{remark}

We next present a result on regularity of the residual $R$. 

\begin{lemma}\label{lemma:regularity-R}
  Assume that $\gamma>4$ and $V_0\in\mathcal D(A^{\gamma'+\frac12})$ for all $\max\left\{\frac54, \frac\gamma2-\frac54\right\}<\gamma'<\frac\gamma2-\frac34$. Suppose that $V=U+R$ is the solution to \eqref{PE-system} 
  , where $U$ solves \eqref{PE-original-linear} with initial data $U(0)=0$, while $R$ solves \eqref{eqn:R} with initial data $R(0)=V_0$. Then for any $T>0$, we have 
   \begin{equation}\label{regularity:R-3}
   \sup\limits_{t\in[0,T]} \|A^{\gamma'+\frac12} R\|^2 + \int_0^T \|A^{\gamma'+1} R\|^2 < \infty, \quad a.s..
   \end{equation}
    Moreover, for an increasing sequence of stopping times $\tau_n$, with $\tau_n\to \infty$,
 \begin{equation}\label{regularity:R-4}
     \mathbb E\left( \sup\limits_{t\in[0,\tau_n]} \|A^{\gamma'+\frac12} R\|^2 + \int_0^{\tau_n} \|A^{\gamma'+1} R\|^2  \right)<\infty.
 \end{equation}
\end{lemma}
\begin{proof}
Recalling the result of Lemma \ref{lemma:regularity-V}, given that $\gamma'+\frac12 >\frac\gamma2-\frac34$ and $\gamma'< \frac\gamma2-\frac34$, we have
\begin{equation}\label{V-2}
    V\in L_{\text{loc}}^2((0,\infty);\mathcal D(A^{\gamma'+\frac12})\cap C([0,\infty); \mathcal D(A^{\gamma'}))\quad a.s..
\end{equation}
The estimate of $\|A^{\gamma'+\frac12} R\|^2$ implies 
\begin{align*}
    \frac d{dt} \|A^{\gamma'+\frac12} R\|^2 + 2\nu_h \|A_h^{\frac12} A^{\gamma'+\frac12} R\|^2  + 2\nu_z \|A_z^{\frac12} A^{\gamma'+\frac12} R\|^2  - 2 \langle A^{\gamma'} B(V,V), A^{\gamma'+1} R\rangle =0.
\end{align*}
As $\gamma'>\frac54$, thanks to Lemma~\ref{lemma:a1} we obtain 
\begin{align*}
    \langle A^{\gamma'} B(V,V), A^{\gamma'+1} R \rangle \leq C \|A^{\gamma'+\frac12} V\| \|A^{\frac12+\frac34+\varepsilon} V\| \|A^{\gamma'+1} R\| \leq C \|A^{\gamma'+\frac12} V\| \|A^{\gamma'} V\| A^{\gamma'+1} R\|,
\end{align*}
where $\varepsilon>0$ is chosen to be small enough such that $\frac54+\varepsilon\leq \gamma'$.
 Applying Young's inequality, we deduce the following estimate for the nonlinear term
\begin{align*}
    |2 \langle A^{\gamma'} B(V,V), A^{\gamma'+1} R\rangle| \leq C \|A^{\gamma'} V\|^2\|A^{\gamma'+\frac12} V\|^2 +  \nu_h \|A_h^{\frac12} A^{\gamma'+\frac12} R\|^2  + \nu_z \|A_z^{\frac12} A^{\gamma'+\frac12} R\|^2.
\end{align*}
Therefore,
\begin{equation}\label{V-3}
    \frac d{dt} \|A^{\gamma'+\frac12} R\|^2 + \nu_h \|A_h^{\frac12} A^{\gamma'+\frac12} R\|^2  + \nu_z \|A_z^{\frac12} A^{\gamma'+\frac12} R\|^2 \leq C\|A^{\gamma'} V\|^2\|A^{\gamma'+\frac12} V\|^2.
\end{equation}
Now, using \eqref{V-2}, we claim that for any $T>0$
\[
\int_0^T \|A^{\gamma'} V\|^2\|A^{\gamma'+\frac12} V\|^2 dt <\infty \quad a.s.,
\]
and \eqref{regularity:R-3} follows from the Gronwall inequality.  

Next, consider the stopping times defined as
\[
\tau_n= \inf\limits_{t\geq 0} \left\{ \sup\limits_{t'\leq t} \|A^{\gamma'} V\|^2 + \int_0^t \|A^{\gamma'+\frac12} V\|^2  dt' > n \right\}.
\]
It is easy to see that $\tau_n$ is increasing. Furthermore, by applying the Gronwall inequality in \eqref{V-3} and taking the expectation we conclude that 
\eqref{regularity:R-4} holds. Lastly, observe that 
$$\mathbb P(\tau_n<T) = \mathbb P\left( \sup\limits_{t'\leq T} \|A^{\gamma'} V\|^2 + \int_0^T \|A^{\gamma'+\frac12} V\|^2  dt' > n \right),$$ 
from \eqref{V-2} one can infer that $\lim\limits_{n\to \infty} \tau_n = \infty.$
The proof is complete. 
\end{proof}

\begin{remark}\label{rmk:R-better-than-U}
    When $U(0)=0$, by Lemma \ref{lemma:regularity-linear}, for any $\gamma'<\frac\gamma2-\frac34$,
    \[
U ,\widetilde U \in  L_{\text{loc}}^2((0,\infty);\mathcal D(A^{\gamma'+\frac12})) \quad \text{and}\quad \overline U , \widehat U \in  L_{\text{loc}}^2((0,\infty);\mathcal D(A^{\gamma'+\frac34})) \quad a.s..
    \]
    On the other hand, from Lemma \ref{lemma:regularity-R}, we have
    \[
  R \in L_{\text{loc}}^2((0,\infty);\mathcal D(A^{\gamma'+1})), \quad a.s..
    \]
    Consequently, the regularity of $R$ is $A^{\frac12}$ higher than $U$ and $\widetilde U$, and is $A^{\frac14}$ higher than $\overline{U}$ and $\widehat U$. This fact is crucial in the proof of the next lemma. However, as discussed in Remark \ref{rmk:linear-improve}, the regularity of $R$ is the same as the horizontal average of $U$, making the subsequent lemma inapplicable to the horizontal average of $U$.
\end{remark}

\begin{lemma}\label{lemma:ratio-1}
  Assume that $\gamma>4$. Suppose that $V$ is the solution to \eqref{PE-system} with an initial condition $V(0) = V_0 \in \mathcal D(A^{\frac12+\gamma'})$ for all $\max\left\{\frac54, \frac\gamma2-\frac54\right\}<\gamma'<\frac\gamma2-\frac34$, and
  $U$ is the solution to \eqref{PE-original-linear} with $U(0)=0$. Then, for any $\alpha > \frac\gamma2-\frac54$, we have
  \begin{align}
      &\lim\limits_{N\to \infty} \frac{\int_0^T \|A^{1+\alpha} V^N\|^2 dt}{\mathbb E\int_0^T \|A^{1+\alpha} U^N\|^2 dt} = 1,  \label{ratio-1}
      \\
      &\lim\limits_{N\to \infty} \frac{\int_0^T \|A^{1+\alpha} \widetilde V^N\|^2 dt}{\mathbb E\int_0^T \|A^{1+\alpha} \widetilde U^N\|^2 dt} = 1, \label{ratio-2}
      \\
      &\lim\limits_{N\to \infty} \frac{\int_0^T \|A_z A^{\alpha} \widetilde V^N\|^2 dt}{\mathbb E\int_0^T \|A_z A^{\alpha} \widetilde U^N\|^2 dt} = 1, \label{ratio-3}
      \\
      &\lim\limits_{N\to \infty} \frac{\int_0^T \|A_h^{\frac12} A_z^{\frac12} A^{\alpha}  \widetilde V^N\|^2 dt}{\mathbb E\int_0^T \|A_h^{\frac12} A_z^{\frac12} A^{\alpha} \widetilde U^N\|^2 dt} = 1, \label{ratio-4}
  \end{align}
  with probability 1. On the other hand, for all $\max\left\{\frac54, \frac\gamma2-1\right\}<\gamma'<\frac\gamma2-\frac34$ and $\alpha>\frac\gamma2-1$, we have
  \begin{align}
      &\lim\limits_{N\to \infty} \frac{\int_0^T \|A^{1+\alpha} \overline V^N\|^2 dt}{\mathbb E\int_0^T \|A^{1+\alpha} \overline U^N\|^2 dt} = 1,  \label{ratio-5}
      \\
    &\lim\limits_{N\to \infty} \frac{\int_0^T \|A_zA^{\alpha} \widehat V^N\|^2 dt}{\mathbb E\int_0^T \|A_zA^{\alpha} \widehat U^N\|^2 dt} = 1 \quad \text{with probability 1}. \label{ratio-6}
  \end{align}
\end{lemma}

\begin{proof}
We will prove these limits one by one, building on the law of large numbers Lemma~\ref{lemma:LLN}. Let $R_{\bk} = \langle R, \phi_{\bk}\rangle$ and $R^N = \sum\limits_{1\leq |\bk|\leq N} R_{\bk}\phi_{\bk}$, and define the projections of $R^N$ as follows:
\[ 
\overline R^N= \sum\limits_{1\leq |\bk|\leq N, k_3=0} R_{\bk}\phi_{\bk},  \quad \widehat R^N= \sum\limits_{1\leq |\bk|\leq N, |\bk'|=\sqrt{q}|k_3|} R_{\bk}\phi_{\bk}, \quad  \widetilde R^N= \sum\limits_{1\leq |\bk|\leq N, k_3\neq 0} R_{\bk}\phi_{\bk}.
\]

\noindent\underline{Proof of \eqref{ratio-1}.}
First note that 
\begin{align*}
    \int_0^T \|A^{1+\alpha} V^N\|^2 dt \leq &\int_0^T \left(\|A^{1+\alpha} U^N\|^2 + \|A^{1+\alpha} R^N\|^2\right) dt 
    \\
    &+ 2 \left(\int_0^T \|A^{1+\alpha} U^N\|^2 dt\right)^{\frac12} \left(\int_0^T \|A^{1+\alpha} R^N\|^2 dt\right)^{\frac12},
\end{align*}
and
\begin{align*}
    \int_0^T \|A^{1+\alpha} V^N\|^2 dt \geq &\int_0^T \left(\|A^{1+\alpha} U^N\|^2 + \|A^{1+\alpha} R^N\|^2\right) dt 
    \\
    &- 2 \left(\int_0^T \|A^{1+\alpha} U^N\|^2 dt\right)^{\frac12} \left(\int_0^T \|A^{1+\alpha} R^N\|^2 dt\right)^{\frac12}.
\end{align*}
Then to get \eqref{ratio-1}, it suffices to show
\[
I_1^N = \frac{\int_0^T \|A^{1+\alpha} U^N\|^2 dt}{\mathbb E\int_0^T \|A^{1+\alpha} U^N\|^2 dt} \to 1 \quad a.s., \text{ and }
I_2^N = \frac{\int_0^T \|A^{1+\alpha} R^N\|^2 dt}{\mathbb E\int_0^T \|A^{1+\alpha} U^N\|^2 dt} \to 0 \quad a.s..
\]
Let 
$$\xi_n:= n^{2\alpha+2}\sum\limits_{\bk\in\mathbb Z^3, |\bk|^2=n} \int_0^T |U_{\bk}|^2(t) dt, \quad b_n:= \sum\limits_{j=1}^n \mathbb E[\xi_j].$$
Notice that there exists some $n\in\mathbb N$ such that $|\bk|^2\neq n, \forall \bk\in \mathbb Z^3$. In such cases, $\xi_n=0$. Thanks to \eqref{order-Uk}, we find
\[
b_n \sim \sum\limits_{\bk\in\mathbb Z^3, 1\leq |\bk|\leq \sqrt n}  |\bk|^{4\alpha+4} |\bk|^{-2\gamma-2} =\sum\limits_{\bk\in\mathbb Z^3, 1\leq |\bk|\leq \sqrt n} |\bk|^{4\alpha-2\gamma+2}.
\]
Given that $\alpha>\frac\gamma2-\frac54$, it follows that $4\alpha-2\gamma+2>-3.$ Using Lemma \ref{lemma:order} we deduce that $b_n\sim n^{2\alpha-\gamma+\frac52}$, and thus $\lim\limits_{n\to \infty} b_n=+\infty.$ 
In view of \eqref{order-Var}, we also obtain
\begin{align*}
    \sum\limits_{n=1}^\infty \frac{Var[\xi_n]}{b_n^2}&\lesssim \sum\limits_{n=1}^\infty\frac{n n^{4\alpha+4} n^{-(2\gamma+3)}}{n^{4\alpha-2\gamma+5} }
     \sim \sum\limits_{n=1}^\infty \frac1{n^3} <\infty,
\end{align*}
where we have used Lemma \ref{counting} to get $\Big|\{\bk\in\mathbb Z^3: |\bk|^2=n\}\Big|\lesssim n.$ Therefore, by applying Lemma \ref{lemma:LLN} we conclude that $\lim\limits_{N\to \infty}I_1^N=1$, a.s..

Regarding the residual part $I^N_2$, \eqref{order-U} gives
\begin{equation*}
    \mathbb E\int_0^T \|A^{1+\alpha} U^N\|^2 dt \sim N^{5+4\alpha-2\gamma}.
\end{equation*}
Using Lemma \ref{lemma:regularity-R}, for $\max\left\{\frac54, \frac\gamma2-\frac54\right\}<\gamma'<\frac\gamma2-\frac34$ we know that
\[
\int_0^T \|A^{1+\gamma'}R\|^2 dt <\infty \quad a.s..
\]
Then, we compute
\[
I_2^N \leq C \frac{\int_0^T \|A^{1+\alpha}R^N\|^2 dt}{N^{5+4\alpha-2\gamma}} \leq C \frac{N^{4(\alpha-\gamma') \int_0^T \|A^{1+\gamma'}R^N\|^2 dt}}{N^{5+4\alpha-2\gamma}} \leq C \frac{\int_0^T \|A^{1+\gamma'}R\|^2 dt}{N^{5+4\gamma'-2\gamma}}.
\]
Since $\gamma'>\frac\gamma2-\frac54$, we conclude that $I_2^N\to 0$ a.s. as $N\to \infty.$ 

\noindent\underline{Proof of \eqref{ratio-2}.} 
The proof follows by similar arguments as for \eqref{ratio-1}. For brevity, we will highlight the key difference, which is the definitions of $\xi_n$ and $b_n$. Here, we have 
\[
\xi_n= n^{2\alpha+2}\sum\limits_{\bk\in\mathbb Z^3, k_3 \neq 0, |\bk|^2=n} \int_0^T \widetilde |U_{\bk}|^2(t) dt,
\]
and 
\begin{align}\label{ratio-2-calculation-1}
    b_n = \sum\limits_{j=1}^n \mathbb E[\xi_j] &\sim \sum\limits_{\bk\in\mathbb Z^3, 1\leq |\bk|\leq \sqrt n} |\bk|^{4\alpha-2\gamma+2} - \sum\limits_{\bk\in\mathbb Z^2, 1\leq |\bk|\leq \sqrt n} |\bk|^{4\alpha-2\gamma+2}\nonumber
\\
&\sim n^{2\alpha-\gamma+\frac52} - n^{2\alpha-\gamma+2} \sim  n^{2\alpha-\gamma+\frac52}.
\end{align}
For the calculations related to $b_n$, we determine its order using Lemma~\ref{lemma:order}. Then, one can conclude $I_1^N \to 1$ and $I_2^N \to 0$ a.s. (where $U^N, R^N$ are now replaced by $\widetilde U^N, \widetilde R^N$) by noting that $\Big|\{\bk\in\mathbb Z^3: k_3\neq 0, |\bk|^2=n\}\Big| \leq \Big|\{\bk\in\mathbb Z^3: |\bk|^2=n\}\Big|\lesssim n.$

\noindent\underline{Proof of \eqref{ratio-3}.} In this case, it is enough to show that 
\[
I_3^N = \frac{\int_0^T \|A_z A^{\alpha} \widetilde U^N\|^2 dt}{\mathbb E\int_0^T \|A_z A^{\alpha} \widetilde U^N\|^2 dt} \to 1 \quad a.s., \ \text{ and } \ 
I_4^N = \frac{\int_0^T \|A_z A^{\alpha} \widetilde R^N\|^2 dt}{\mathbb E\int_0^T \|A_z A^{\alpha} \widetilde U^N\|^2 dt} \to 0 \quad a.s..
\]
Let 
$$
\xi_n= |k_3|^4|\bk|^{4\alpha}\sum\limits_{\bk\in\mathbb Z^3, k_3\neq 0, |\bk|^2=n} \int_0^T \widetilde |U_{\bk}|^2(t) dt, \qquad b_n= \sum\limits_{j=1}^n \mathbb E[\xi_j].
$$ 
Using estimates \eqref{order-Uktilde}, \eqref{ratio-2-calculation-1}, and  \eqref{eqn:same-order}, we arrive at
\begin{align*}
    b_n \sim \sum\limits_{\bk\in\mathbb Z^3, k_3\neq 0, 1\leq |\bk|\leq \sqrt n} |k_3|^4 |\bk|^{4\alpha} |\bk|^{-2\gamma-2} \sim \sum\limits_{\bk\in\mathbb Z^3, k_3\neq 0, 1\leq |\bk|\leq \sqrt n}  |\bk|^{4+4\alpha} |\bk|^{-2\gamma-2}\sim n^{2\alpha-2\gamma+2}.
\end{align*}
By noting that
\[
\Var[\xi_n] \leq \Var\left[|\bk|^{4+4\alpha}\sum\limits_{\bk\in\mathbb Z^3, k_3\neq 0, |\bk|^2=n} \int_0^T \widetilde |U_{\bk}|^2(t) dt\right],
\]
one can follow a procedure similar to that used for \eqref{ratio-1} and \eqref{ratio-2} to obtain the desired result.

\noindent\underline{Proof of \eqref{ratio-4}.} The proof follows a similar approach to that of \eqref{ratio-3}, and we omit its details for the sake of brevity.

\noindent\underline{Proof of \eqref{ratio-5}.} We aim to show
\[
\overline I_1^N := \frac{\int_0^T \|A^{1+\alpha} \overline U^N\|^2 dt}{\mathbb E\int_0^T \|A^{1+\alpha}\overline U^N\|^2 dt} \to 1 \quad a.s., \ \text{ and } \
\overline I_2^N := \frac{\int_0^T \|A^{1+\alpha}\overline R^N\|^2 dt}{\mathbb E\int_0^T \|A^{1+\alpha}\overline U^N\|^2 dt} \to 0 \quad a.s..
\]
Let 
$$\overline\xi_n:= n^{2\alpha+2}\sum\limits_{\bk\in\mathbb Z^3, |\bk|^2=n,k_3=0} \int_0^T \overline |U_{\bk}|^2(t) dt, \quad \overline b_n:= \sum\limits_{j=1}^n \mathbb E[\overline\xi_j].$$
Then using \eqref{order-Ukbar} yields 
\[
\overline b_n \sim \sum\limits_{\bk\in\mathbb Z^3, 1\leq |\bk|\leq \sqrt n, k_3=0}  |\bk|^{4\alpha+4} |\bk|^{-2\gamma-2} =\sum\limits_{\bk\in\mathbb Z^3, 1\leq |\bk|\leq \sqrt n, k_3=0} |\bk|^{4\alpha-2\gamma+2}.
\]
As $\alpha>\frac\gamma2-1$, we have $4\alpha-2\gamma+2>-2$ and $\overline{b}_n \sim n^{2+2\alpha-\gamma}$ due to Lemma~\ref{lemma:order}. Consequently,  $\lim\limits_{n\to \infty} \overline b_n=+\infty.$ 
Next, the estimates in \eqref{order-Var} imply 
\begin{align*}
    \sum\limits_{n=1}^\infty \frac{Var[\overline\xi_n]}{\overline b_n^2}&\sim \sum\limits_{n=1}^\infty\frac{n^{\frac12}n^{4\alpha+4} n^{-(2\gamma+3)}}{n^{4+4\alpha-2\gamma}}  \sim \sum\limits_{n=1}^\infty \frac1{n^\frac52} <\infty,
\end{align*}
where we have used the fact that $\Big|\{\bk\in\mathbb Z^3, k_3=0: |\bk|^2=n\}\Big| = \Big|\{\bk\in\mathbb Z^2: |\bk|^2=n\}\Big|\lesssim \sqrt n$. Therefore,  $\lim\limits_{N\to \infty}\overline I_1^N=1$ follows by applying Lemma~\ref{lemma:LLN}. 

For the residual part, the computation
\[
\overline I_2^N \leq C \frac{\int_0^T \|A^{1+\alpha}\overline R^N\|^2 dt}{N^{4+4\alpha-2\gamma}} \leq C \frac{N^{4(\alpha-\gamma') \int_0^T \|A^{1+\gamma'}\overline R^N\|^2 dt}}{N^{4+4\alpha-2\gamma}} \leq C \frac{\int_0^T \|A^{1+\gamma'}R\|^2 dt}{N^{4+4\gamma'-2\gamma}},
\]
and the fact $\gamma'>\frac\gamma2-1$ lead to  $\overline I_2^N\to 0$, as $N\to \infty.$

\noindent\underline{Proof of \eqref{ratio-6}.} The proof is similar to the one of \eqref{ratio-5}, and for the sake of brevity we omit it here.
\end{proof}

\begin{remark}\label{rmk:hor-ave-not-work}
    The results in \eqref{ratio-1}--\eqref{ratio-4} require that  $\max\left\{\frac54, \frac\gamma2-\frac54\right\}<\gamma'<\frac\gamma2-\frac34$ and $\alpha>\frac\gamma2-\frac54$, while \eqref{ratio-5}--\eqref{ratio-6} require stricter conditions $\max\left\{\frac54, \frac\gamma2-1\right\}<\gamma'<\frac\gamma2-\frac34$ and $\alpha>\frac\gamma2-1$. The main reason is that $V,U,\widetilde V,\widetilde U$ are three-dimensional functions and $R$ is $A^{\frac12}$ smoother than $V,U,\widetilde V,\widetilde U$, while $\overline{V},\overline{U},\widehat{V},\widehat{U}$ are essentially two-dimensional functions and $R$ is only $A^{\frac14}$ smoother than $V,U,\widetilde V,\widetilde U$. 

 This observation is also relevant to why one cannot employ the same strategy of taking the horizontal average and build estimators for $\nu_z$ as was done for $\nu_h$ through the solution to \eqref{PE-barotropic}. We point out that with $V_H$ being the horizontal average and $U_H$ its linear part, then both $V_H$ and $U_H$ are one-dimensional functions; that is, they lost two dimensions. Although the equation for $V_H$ remains self-contained and contains solely $\nu_z$, the estimates in Lemma~\ref{lemma:ratio-1} applied to $V_H^N$ and $U_H^N$ (the projections of $V_H$ and $U_H$ onto the first $N$ Fourier modes) will immediately break down. The limit $\lim\limits_{N\to \infty} \frac{\int_0^T \|A^{1+\alpha} V_H^N\|^2 dt}{\mathbb E\int_0^T \|A^{1+\alpha} U_H^N\|^2 dt} = 1$ will no longer hold, primarily because the condition $\max\left\{\frac54, \frac\gamma2-\frac34\right\}<\gamma'<\frac\gamma2-\frac34$ can never be satisfied. The fundamental reason behind this is that $R$ has exactly the same regularity as $U_H$; see Remark~\ref{rmk:linear-improve} and Remark~\ref{rmk:R-better-than-U}).

However, if the original system \eqref{PE-system} is in dimension two, then taking the horizontal average indeed works, as the reduction in dimensions is only by one.
\end{remark}

\section{Proof of the Main Theorems}\label{sec:proof}

In this section, we present the proofs of Theorem \ref{thm:consistency} and Theorem \ref{thm:normality}. To perform the analysis, we first establish the following representations for the proposed estimators $\nu_{h1}^N$, $\nu_{h2}^N$, $\nu_{h3}^N$, $\nu_{z1}^N$,$\nu_{z2}^N$, $\nu_{z3}^N$, and $\widehat{\nu_{z1}^N}$.

Using the dynamics \eqref{PE-original-finitemode} of $\overline V$,  we first derive that 
\begin{align}\label{nuh1N}
    \nu^N_{h1} &= -\frac{\int_0^T  \langle A_h^{1+\alpha} \overline{V}^N, d\overline{V}^N \rangle + \int_0^T  \langle A_h^{1+\alpha} \overline{V}^N , \mathcal P_h\overline{B_N(V,V)} \rangle dt}{\int_0^T \|A_h^{1+\frac{\alpha}{2}} \overline{V}^N\|^2 dt} = \nu_h - \frac{\int_0^T  \langle A_h^{1+\alpha} \overline{V}^N, P_N \overline{\sigma dW} \rangle  dt }{\int_0^T \|A_h^{1+\frac{\alpha}{2}} \overline{V}^N\|^2 dt} \nonumber
    \\
    &= \nu_h - \frac{\int_0^T  \langle A_h^{1+\alpha-\frac\gamma2} \overline{V}^N, \sigma_0\sum\limits_{k_3=0,1\leq|\bk|\leq N} c_{\bk} \phi_{\bk} dW_{\bk} \rangle }{\int_0^T \|A_h^{1+\frac{\alpha}{2}} \overline{V}^N\|^2 dt}.
\end{align}
Similarly, we get 
\begin{align*}
   \nu^N_{z1}=& -\frac1{\int_0^T \|A_z A^{\frac{\alpha}{2}} \widetilde{V}^N\|^2 dt}\Bigg(\int_0^T  \langle A_z A^{\alpha} \widetilde{V}^N, d\widetilde{V}^N \rangle +     \int_0^T \left\langle A_z A^{\alpha} \widetilde{V}^N , \widetilde{B_N(V,V)} \right\rangle dt
   \\
   &\hspace{2cm}-\frac{\int_0^T  \langle A_h^{1+\alpha} \overline{V}^N, d\overline{V}^N \rangle + \int_0^T  \langle A_h^{1+\alpha} \overline{V}^N , \mathcal P_h\overline{B_N(V,V)} \rangle dt}{\int_0^T \|A_h^{1+\frac{\alpha}{2}} \overline{V}^N\|^2 dt}\int_0^T \langle A_h A^{\alpha} \widetilde{V}^N,  A_z  \widetilde{V}^N \rangle dt\Bigg) 
   \\
   =& \nu_z -\frac{\int_0^T  \langle A_z A^{\alpha-\frac\gamma2} \widetilde{V}^N, \sigma_0\sum\limits_{k_3\neq0,1\leq|\bk|\leq N} c_{\bk} \phi_{\bk} dW_{\bk} \rangle}{\int_0^T \|A_z A^{\frac{\alpha}{2}} \widetilde{V}^N\|^2 dt}
   \\
   &+ \frac{\int_0^T  \langle A_h^{1+\alpha-\frac\gamma2} \overline{V}^N, \sigma_0\sum\limits_{k_3=0,1\leq|\bk|\leq N} c_{\bk} \phi_{\bk} dW_{\bk} \rangle \int_0^T \langle A_h A^{\alpha} \widetilde{V}^N,  A_z  \widetilde{V}^N \rangle dt}{\int_0^T \|A_h^{1+\frac{\alpha}{2}} \overline{V}^N\|^2 dt\int_0^T \|A_z A^{\frac{\alpha}{2}} \widetilde{V}^N\|^2 dt}.
\end{align*}

Likewise, for $\nu^N_{h2}$ and $\nu^N_{z2}$ we write
\begin{align*}
    \nu^N_{h2} = \nu^N_{h1} + \frac{\int_0^T  \langle A_h^{1+\alpha} \overline{V}^N , \overline{B_N(V,V)} -\overline{B_N(V^N,V^N)} \rangle  dt}{\int_0^T \|A_h^{1+\frac{\alpha}{2}} \overline{V}^N\|^2 dt}, 
\end{align*}
and 
\begin{align*}
   \nu^N_{z2}=& \nu^N_{z1} + \frac{\int_0^T  \langle A_z A^{\alpha} \widetilde{V}^N, \widetilde{B_N(V,V)} - \widetilde{B_N(V^N,V^N)} \rangle dt}{\int_0^T \|A_z A^{\frac{\alpha}{2}} \widetilde{V}^N\|^2 dt}
   \\
   &-\frac{\int_0^T  \langle A_h^{1+\alpha} \overline{V}^N, \overline{B_N(V,V)} - \overline{B_N(V^N,V^N)} \rangle dt \int_0^T \langle A_h A^{\alpha} \widetilde{V}^N,  A_z  \widetilde{V}^N \rangle dt}{\int_0^T \|A_h^{1+\frac{\alpha}{2}} \overline{V}^N\|^2 dt\int_0^T \|A_z A^{\frac{\alpha}{2}} \widetilde{V}^N\|^2 dt}.
\end{align*}
Moreover, for $\nu^N_{h3}$ and $\nu^N_{z3}$, we get 
\begin{align*}
    \nu^N_{h3} = \nu^N_{h1} + \frac{\int_0^T  \langle A_h^{1+\alpha} \overline{V}^N , \overline{B_N(V,V)}  \rangle dt }{\int_0^T \|A_h^{1+\frac{\alpha}{2}} \overline{V}^N\|^2 dt}, 
\end{align*}
and 
\begin{align*}
   \nu^N_{z3}= \nu^N_{z1} + \frac{\int_0^T  \langle A_z A^{\alpha} \widetilde{V}^N, \widetilde{B_N(V,V)}  \rangle dt}{\int_0^T \|A_z A^{\frac{\alpha}{2}} \widetilde{V}^N\|^2 dt}
   -\frac{\int_0^T  \langle A_h^{1+\alpha} \overline{V}^N, \overline{B_N(V,V)} \rangle dt \int_0^T \langle A_h A^{\alpha} \widetilde{V}^N,  A_z  \widetilde{V}^N \rangle dt}{\int_0^T \|A_h^{1+\frac{\alpha}{2}} \overline{V}^N\|^2 dt\int_0^T \|A_z A^{\frac{\alpha}{2}} \widetilde{V}^N\|^2 dt}.
\end{align*}

Regarding the representation of $\widehat{\nu^N_{z1}}$, from $\nu_{z1}^N$ we obtain 
\begin{align}\label{hatnuz1N}
     \widehat{\nu^N_{z1}}=& -\frac1{\int_0^T \|A_z A^{\frac{\alpha}{2}} \widehat{V}^N\|^2 dt}\Bigg(\int_0^T  \langle A_z A^{\alpha} \widehat{V}^N, d\widehat{V}^N \rangle +     \int_0^T \left\langle A_z A^{\alpha} \widehat{V}^N , \widehat{B_N(V,V)} \right\rangle dt \nonumber
   \\
   &\hspace{2cm}-\frac{\int_0^T  \langle A_h^{1+\alpha} \overline{V}^N, d\overline{V}^N \rangle + \int_0^T  \langle A_h^{1+\alpha} \overline{V}^N , \mathcal P_h\overline{B_N(V,V)} \rangle dt}{\int_0^T \|A_h^{1+\frac{\alpha}{2}} \overline{V}^N\|^2 dt}\int_0^T \langle A_h A^{\alpha} \widehat{V}^N,  A_z  \widehat{V}^N \rangle dt\Bigg) \nonumber
   \\
   =&\nu_z - \frac{\int_0^T  \langle A_z A^{\alpha-\frac\gamma2} \widehat{V}^N, \sigma_0\sum\limits_{1\leq|\bk|\leq N, |\bk'|=\sqrt{q}|k_3|} c_{\bk} \phi_{\bk} dW_{\bk} \rangle}{\int_0^T \|A_z A^{\frac{\alpha}{2}} \widehat{V}^N\|^2 dt} \nonumber
   \\
   &+ \frac{\int_0^T  \langle A_h^{1+\alpha-\frac\gamma2} \overline{V}^N, \sigma_0\sum\limits_{k_3=0,1\leq|\bk|\leq N} c_{\bk} \phi_{\bk} dW_{\bk} \rangle \int_0^T \langle A_h A^{\alpha} \widehat{V}^N,  A_z  \widehat{V}^N \rangle dt}{\int_0^T \|A_h^{1+\frac{\alpha}{2}} \overline{V}^N\|^2 dt\int_0^T \|A_z A^{\frac{\alpha}{2}} \widehat{V}^N\|^2 dt} \nonumber
   \\
   =&\nu_z - \frac{\int_0^T  \langle A_z A^{\alpha-\frac\gamma2} \widehat{V}^N, \sigma_0\sum\limits_{1\leq|\bk|\leq N, |\bk'|=\sqrt{q}|k_3|} c_{\bk} \phi_{\bk} dW_{\bk} \rangle}{\int_0^T \|A_z A^{\frac{\alpha}{2}} \widehat{V}^N\|^2 dt} + q\frac{\int_0^T  \langle A_h^{1+\alpha-\frac\gamma2} \overline{V}^N, \sigma_0\sum\limits_{k_3=0,1\leq|\bk|\leq N} c_{\bk} \phi_{\bk} dW_{\bk} \rangle}{\int_0^T \|A_h^{1+\frac{\alpha}{2}} \overline{V}^N\|^2 dt},
\end{align}
where we have used the fact that $\int_0^T \langle A_h A^{\alpha} \widehat{V}^N,  A_z  \widehat{V}^N \rangle dt = q\int_0^T \|A_z A^{\frac{\alpha}{2}} \widehat{V}^N\|^2 dt$ because $A_h \widehat V^N = q A_z \widehat V^N $.

\subsection{Consistency of the estimators}\label{sec:consistency}

It is enough to show that each stochastic term in the derived representations converges to zero. We  begin by examining the consistency of $\nu_{h1}^N$, $\nu_{z1}^N$, and $\widehat{\nu_{z1}^N}$. 
\begin{lemma}\label{lemma:stochastic-term}
Assume that $\gamma>4$. Suppose that $V$ is the solution to \eqref{PE-system} with an initial condition $V(0) = V_0 \in \mathcal D(A^{\frac12+\gamma'})$ for all $\max\left\{\frac54, \frac\gamma2-1\right\}<\gamma'<\frac\gamma2-\frac34$, and
  $U$ is the solution to \eqref{PE-original-linear} with $U(0)=0$. Then for any $\alpha > \gamma-2$, we have
  
 (i) For every $\overline\delta_1,\widehat \delta_1<\min\{4+2\alpha-2\gamma,\frac32\}$ and $\widetilde\delta_1<\min\{5+2\alpha-2\gamma,2\}$,
 \begin{equation}\label{linear-to-zero-1}
    \lim\limits_{N\to \infty} N^{\overline\delta_1} \frac{\int_0^T  \langle A_h^{1+\alpha-\frac\gamma2} \overline{U}^N, \sum\limits_{k_3=0,1\leq|\bk|\leq N} c_{\bk} \phi_{\bk} dW_{\bk} \rangle}{\int_0^T \|A_h^{1+\frac{\alpha}{2}} \overline{V}^N\|^2 dt} =0 \quad a.s.,
 \end{equation}
 \begin{equation}\label{linear-to-zero-2}
    \lim\limits_{N\to \infty} N^{\widehat\delta_1} \frac{\int_0^T  \langle A_z A^{\alpha-\frac\gamma2} \widehat{U}^N, \sum\limits_{1\leq|\bk|\leq N, |\bk'|=\sqrt{q}|k_3|} c_{\bk} \phi_{\bk} dW_{\bk} \rangle}{\int_0^T \|A_z A^{\frac{\alpha}{2}} \widehat{V}^N\|^2 dt} =0 \quad a.s.,
 \end{equation}
 and
 \begin{equation}\label{linear-to-zero-3}
    \lim\limits_{N\to \infty} N^{\widetilde\delta_1} \frac{\int_0^T  \langle A_z A^{\alpha-\frac\gamma2} \widetilde{U}^N, \sum\limits_{k_3\neq 0,1\leq|\bk|\leq N} c_{\bk} \phi_{\bk} dW_{\bk} \rangle}{\int_0^T \|A_z A^{\frac{\alpha}{2}} \widetilde{V}^N\|^2 dt} =0 \quad a.s..
 \end{equation}
 
 (ii) For every $\overline\delta_2,\widehat\delta_2<\min\{4+2\alpha-2\gamma,\frac52 \}$ and $\widetilde\delta_2<\min\{5+2\alpha-2\gamma,\frac72 \}$,
 \begin{equation}\label{residual-to-zero-1}
    \lim\limits_{N\to \infty} N^{\overline\delta_2} \frac{\int_0^T  \langle A_h^{1+\alpha-\frac\gamma2} \overline{R}^N, \sum\limits_{k_3=0,1\leq|\bk|\leq N} c_{\bk} \phi_{\bk} dW_{\bk} \rangle}{\int_0^T \|A_h^{1+\frac{\alpha}{2}} \overline{V}^N\|^2 dt} =0,
 \end{equation}
  \begin{equation}\label{residual-to-zero-2}
    \lim\limits_{N\to \infty} N^{\widehat\delta_2} \frac{\int_0^T  \langle A_z A^{\alpha-\frac\gamma2} \widehat{R}^N, \sum\limits_{1\leq|\bk|\leq N, |\bk'|=\sqrt{q}|k_3|} c_{\bk} \phi_{\bk} dW_{\bk} \rangle}{\int_0^T \|A_z A^{\frac{\alpha}{2}} \widehat{V}^N\|^2 dt} =0 ,
 \end{equation}
 and
 \begin{equation}\label{residual-to-zero-3}
    \lim\limits_{N\to \infty} N^{\widetilde\delta_2} \frac{\int_0^T  \langle A_z A^{\alpha-\frac\gamma2} \widetilde{R}^N, \sum\limits_{k_3\neq 0,1\leq|\bk|\leq N} c_{\bk} \phi_{\bk} dW_{\bk} \rangle}{\int_0^T \|A_z A^{\frac{\alpha}{2}} \widetilde{V}^N\|^2 dt} =0, 
 \end{equation}
in probability.
\end{lemma}

\begin{proof}
Since $\frac\alpha2>\frac\gamma2-1$, thanks to Lemma \ref{lemma:ratio-1} and Lemma \ref{lemma:order-linear}, we deduce that as $N\to\infty$,
\begin{align*}
    &\int_0^T \|A_h^{1+\frac{\alpha}{2}} \overline{V}^N\|^2 dt \sim \mathbb E \int_0^T \|A_h^{1+\frac{\alpha}{2}} \overline{U}^N\|^2 dt \sim N^{4+2\alpha-2\gamma},
    \\
    &\int_0^T \|A_z A^{\frac{\alpha}{2}} \widehat{V}^N\|^2 dt \sim \mathbb E \int_0^T \|A_z A^{\frac{\alpha}{2}} \widehat{U}^N\|^2 dt \sim N^{4+2\alpha-2\gamma},
    \\
    &\int_0^T \|A_z A^{\frac{\alpha}{2}} \widetilde{V}^N\|^2 dt \sim \mathbb E\int_0^T \|A_z A^{\frac{\alpha}{2}} \widetilde{U}^N\|^2 dt \sim N^{5+2\alpha-2\gamma}.
\end{align*}
Therefore, it suffices to show that each of the next terms converges to zero, as $N\to\infty$:
\begin{align*}
    &\overline{J}_N^1= \frac{\int_0^T  \langle A_h^{1+\alpha-\frac\gamma2} \overline{U}^N, \sum\limits_{k_3=0,1\leq|\bk|\leq N} c_{\bk} \phi_{\bk} dW_{\bk} \rangle}{N^{4+2\alpha-2\gamma-\overline\delta_1}}= \frac{\sum\limits_{k_3=0,1\leq|\bk|\leq N}|\bk|^{2+2\alpha-\gamma}\int_0^T \overline{U}_{\bk}\cdot c_{\bk} dW_{\bk}}{N^{4+2\alpha-2\gamma-\overline\delta_1}},
    \\
    &\widehat{J}_N^1= \frac{\int_0^T  \langle A_z A^{\alpha-\frac\gamma2} \widehat{U}^N, \sum\limits_{1\leq|\bk|\leq N,|\bk'|=\sqrt{q}|k_3|} c_{\bk} \phi_{\bk} dW_{\bk} \rangle}{N^{4+2\alpha-2\gamma-\widehat\delta_1}}= \frac{\sum\limits_{1\leq|\bk|\leq N,|\bk'|=\sqrt{q}|k_3|}\frac12|\bk|^{2+2\alpha-\gamma}\int_0^T \widehat{U}_{\bk}\cdot c_{\bk} dW_{\bk}}{N^{4+2\alpha-2\gamma-\widehat\delta_1}},
    \\
    &\widetilde{J}_N^1= \frac{\int_0^T  \langle A_z A^{\alpha-\frac\gamma2} \widetilde{U}^N, \sum\limits_{k_3\neq 0,1\leq|\bk|\leq N} c_{\bk} \phi_{\bk} dW_{\bk} \rangle}{N^{5+2\alpha-2\gamma-\widetilde\delta_1}}= \frac{\sum\limits_{k_3\neq 0,1\leq|\bk|\leq N}|k_3|^2|\bk|^{2\alpha-\gamma}\int_0^T \widetilde{U}_{\bk}\cdot c_{\bk} dW_{\bk}}{N^{5+2\alpha-2\gamma-\widetilde\delta_1}},
\end{align*}
and
\begin{align*}
    &\overline{J}_N^2 = \frac{\int_0^T  \langle A_h^{1+\alpha-\frac\gamma2} \overline{R}^N, \sum\limits_{k_3=0,1\leq|\bk|\leq N} c_{\bk} \phi_{\bk} dW_{\bk} \rangle}{N^{4+2\alpha-2\gamma-\overline\delta_2}},
    \\
    &\widehat{J}_N^2 = \frac{\int_0^T  \langle A_z A^{\alpha-\frac\gamma2} \widehat{R}^N, \sum\limits_{1\leq|\bk|\leq N, |\bk'|=\sqrt{q}|k_3|} c_{\bk} \phi_{\bk} dW_{\bk} \rangle}{N^{4+2\alpha-2\gamma-\widehat\delta_2}},
    \\
    &\widetilde{J}_N^2= \frac{\int_0^T  \langle A_z A^{\alpha-\frac\gamma2} \widetilde{R}^N, \sum\limits_{k_3\neq 0,1\leq|\bk|\leq N} c_{\bk} \phi_{\bk} dW_{\bk} \rangle}{N^{5+2\alpha-2\gamma-\widetilde\delta_2}}
\end{align*}
 We will establish these limits one by one. 

For $\overline{J}_N^1$, we define $\overline{\xi}_n= n^{1+\alpha-\frac\gamma2} \sum\limits_{|\bk|^2=n,k_3=0} \int_0^T \overline{U}_{\bk}\cdot c_{\bk} dW_{\bk}$ and $\overline{b}_n=n^{2+\alpha-\gamma-\frac{\overline\delta_1}2}$. Note that under the condition on $\overline{\delta}_1$ we have $\lim\limits_{n\to\infty} \overline{b}_n = \infty.$ Using the It\^{o}'s isometry and \eqref{order-Ukbar}, since $\mathbb E[\overline{\xi}_n]=0$, we obtain
\begin{align*}
    \Var[\overline{\xi}_n]  =  \mathbb E[\overline{\xi}_n^2] = n^{2+2\alpha-\gamma}  \sum\limits_{|\bk|^2=n,k_3=0} \mathbb E \int_0^T |\overline{U}_{\bk}|^2 dt \lesssim n^{2+2\alpha-\gamma} \sqrt{n} n^{-\gamma-1} \sim n^{\frac32+2\alpha-2\gamma},
\end{align*}
where we have also used the fact that $\Big|\{\bk\in\mathbb Z^3, k_3=0: |\bk|^2=n\}\Big| = \Big|\{\bk\in\mathbb Z^2: |\bk|^2=n\}\Big|\lesssim \sqrt n.$
Therefore, in the regime  $\overline{\delta}_1<\frac32$, we have
\begin{align*}
    \sum\limits_{n=1}^\infty \frac{\Var[\overline{\xi}_n]}{\overline{b}_n^2} \lesssim \sum\limits_{n=1}^\infty \frac{n^{\frac32+2\alpha-2\gamma}}{n^{4+2\alpha-2\gamma-\overline{\delta}_1}} =  \sum\limits_{n=1}^\infty \frac{1}{n^{\frac52-\overline{\delta}_1}}<\infty.
\end{align*}
By invoking Lemma \ref{lemma:LLN}, we conclude that $\lim\limits_{N\to\infty} \overline{J}_N^1 = 0$ a.s..

The proof of $\widehat{J}_N^1$ follows by a similar approach, given the fact that  $$\Big|\{\bk\in\mathbb Z^3: |\bk|^2=n, |\bk'|=\sqrt{q}|k_3|\}\Big| = 2 \Big|\{\bk'\in\mathbb Z^2: |\bk'|^2=\frac{q}{q+1}n\}\Big|\lesssim \sqrt n.$$

For $\widetilde J_N^1$, 
we put $\widetilde{\xi}_n= |k_3|^{2} |\bk|^{2\alpha-\gamma} \sum\limits_{|\bk|^2=n,k_3\neq 0} \int_0^T \widetilde{U}_{\bk}\cdot c_{\bk} dW_{\bk}$ and $\widetilde{b}_n=n^{\frac52+\alpha-\gamma-\frac{\widetilde\delta_1}2}$. Again, under the condition on $\widetilde{\delta}_1$, we get $\lim\limits_{n\to\infty} \widetilde{b}_n = \infty.$  The It\^{o}'s isometry, estimate \eqref{order-Uktilde}, and the fact that $\mathbb E[\widetilde{\xi}_n]=0$ together imply 
\begin{align*}
    \Var[\widetilde{\xi}_n]  =  \mathbb E[\widetilde{\xi}_n^2] \leq  n^{2+2\alpha-\gamma}  \sum\limits_{|\bk|^2=n,k_3\neq 0} \mathbb E \int_0^T |\widetilde{U}_{\bk}|^2 dt \sim n^{2+2\alpha-\gamma} n n^{-\gamma-1} \sim n^{2+2\alpha-2\gamma},
\end{align*}
where once again, we have utilized Lemma \ref{counting}.
Thus, when $\widetilde{\delta}_1<2$, we have
\begin{align*}
    \sum\limits_{n=1}^\infty \frac{\Var[\widetilde{\xi}_n]}{\widetilde{b}_n^2} \lesssim \sum\limits_{n=1}^\infty \frac{n^{2+2\alpha-2\gamma}}{n^{5+2\alpha-2\gamma-\widetilde{\delta}_1}} =  \sum\limits_{n=1}^\infty \frac{1}{n^{3-\widetilde{\delta}_1}}<\infty,
\end{align*}
and conclude, invoking Lemma~\ref{lemma:LLN}, that $\lim\limits_{N\to\infty} \widetilde{J}_N^1 = 0$ with probability one. 

As for $\overline{J}_N^2$, let $\overline{r}_{\bk} = \langle \overline{R},\phi_{\bk}\rangle$, and we define, for any stopping time $\tau$,
\[
\overline{\zeta}_n^\tau= n^{1+\alpha-\frac\gamma2} \sum\limits_{|\bk|^2=n, k_3=0} \int_0^{\tau}\overline{r}_{\bk} dW_{\bk},
\]
and the sequence $\overline{c}_n= n^{2+\alpha-\gamma-\frac{\overline{\delta}_2}2}$. With the given condition on $\overline{\delta}_2$, it is clear that $\overline c_n$ is increasing and $\lim\limits_{n\to\infty}\overline c_n = \infty$. Consequently, we can infer, for any stopping time $\tau$ that satisfies $\Var[\overline\zeta_n^{\tau}]<\infty$,
\begin{align*}
    &\sum\limits_{n=1}^\infty \frac{\Var[\overline\zeta_n^{\tau}]}{\overline{c}_n^2} = \sum\limits_{n=1}^\infty \frac{n^{2+2\alpha-\gamma}}{n^{4+2\alpha-2\gamma-\overline{\delta}_2}} \sum\limits_{|\bk|^2=n, k_3=0} \mathbb E \int_0^{\tau}|\overline{r}_{\bk}|^2 dt 
    \\
    = &\sum\limits_{n=1}^\infty  \sum\limits_{|\bk|^2=n, k_3=0} |\bk|^{2\gamma+2\overline{\delta}_2-4} \mathbb E \int_0^{\tau}|\overline{r}_{\bk}|^2 dt = \mathbb E \int_0^{\tau} \|A^{\frac\gamma2+\frac{\overline{\delta}_2}2 -1} \overline{R}\|^2 dt \leq \mathbb E \int_0^{\tau} \|A^{\frac\gamma2+\frac{\overline{\delta}_2}2 -1} R\|^2 dt.
\end{align*}
Using the result in \eqref{regularity:R-4} and taking $\tau_m$ as in Lemma \ref{lemma:regularity-R}, under the  condition satisfied by $\overline{\delta}_2$, we get that 
\[
\sum\limits_{n=1}^\infty  \frac{\Var[\overline\zeta_k^{T\wedge \tau_m}]}{\overline{c}_n^2} <\infty.
\]
By applying Lemma \ref{lemma:LLN}, since $\{\overline\zeta_{n}^{\tau}\}_{n\in\mathbb N}$ are uncorrelated,  for each $m\in\mathbb N$, we establish that 
\[
 \lim\limits_{N\to\infty} \frac{\int_0^{T\wedge \tau_m}  \langle A_h^{1+\alpha-\frac\gamma2} \overline{R}^N, \sum\limits_{k_3=0,1\leq|\bk|\leq N} c_{\bk} \phi_{\bk} dW_{\bk} \rangle}{N^{4+2\alpha-2\gamma-\overline\delta_2}}=0 \quad \text{in probability}. 
\]
As $\tau_m \nearrow \infty$ and  $\mathbb{P}(\cup_m \{\tau_m >T\}) = 1$, we conclude that $\lim\limits_{N\to \infty} \overline{J}_N^2 = 0$ in probability. 

The proof of $\widehat{J}_N^2$ follows by similar arguments. Moving on to $\widetilde{J}_N^2$, we let $\widetilde{r}_{\bk} = \langle \widetilde{R},\phi_{\bk}\rangle$ and define, for any stopping time $\tau$,
\[
\widetilde{\zeta}_n^\tau= \sum\limits_{|\bk|^2=n, k_3\neq 0}  |k_3|^2 |\bk|^{2\alpha-\gamma} \int_0^{\tau}\widetilde{r}_{\bk} dW_{\bk}.
\]
and also consider the sequence $\widetilde{c}_n= n^{\frac52+\alpha-\gamma-\frac{\widetilde{\delta}_2}2}$. Again, due to the imposed condition on $\widetilde{\delta}_2$, we know that $\widetilde c_n$ is increasing and $\lim\limits_{n\to\infty}\widetilde c_n = \infty$; and up to any stopping time $\tau$ such that $\Var[\widetilde\zeta_n^{\tau}]<\infty$, one has
\begin{align*}
    &\sum\limits_{n=1}^\infty \frac{\Var[\widetilde\zeta_n^{\tau}]}{\widetilde{c}_n^2} \leq \sum\limits_{n=1}^\infty \frac{n^{2+2\alpha-\gamma}}{n^{5+2\alpha-2\gamma-\widetilde{\delta}_2}} \sum\limits_{|\bk|^2=n, k_3\neq 0} \mathbb E \int_0^{\tau}|\widetilde{r}_{\bk}|^2 dt 
    \\
    = &\sum\limits_{n=1}^\infty  \sum\limits_{|\bk|^2=n, k_3\neq 0} |\bk|^{2\gamma+2\widetilde{\delta}_2-6} \mathbb E \int_0^{\tau}|\widetilde{r}_{\bk}|^2 dt = \mathbb E \int_0^{\tau} \|A^{\frac\gamma2+\frac{\widetilde{\delta}_2}2 -\frac32} \widetilde{R}\|^2 dt \leq \mathbb E \int_0^{\tau} \|A^{\frac\gamma2+\frac{\widetilde{\delta}_2}2 -\frac32} R\|^2 dt.
\end{align*}
Following the same reasoning as applied to $\overline{J}_N^2$, and taking into account the condition imposed on $\widetilde\delta_2$, we can assert that $\sum\limits_{n=1}^\infty  \frac{\Var[\zeta_k^{T\wedge \tau_m}]}{\widetilde{c}_n^2} <\infty$ holds for any $\tau_m$ defined as per Lemma~\ref{lemma:regularity-R}. Consequently, we arrive at the conclusion that  $\lim\limits_{N\to \infty} \widetilde{J}_N^2 = 0$ in probability.
\end{proof}

Now we are ready to prove the first part of Theorem~\ref{thm:consistency}.

\begin{proposition}\label{prop:consistent-1}
For $\gamma>4$ and $\alpha>\gamma-2$, $\nu^N_{h1}$, $\nu^N_{z1}$, and $\widehat{\nu^N_{z1}}$ are weakly consistent estimators of the true parameters $\nu_h$ and $\nu_z$, respectively.
\end{proposition}
\begin{proof}
By taking $\overline\delta_1=\overline\delta_2 = \widehat\delta_1 = \widehat\delta_2 = \widetilde\delta_1 = \widetilde\delta_2=0$ in Lemma \ref{lemma:stochastic-term}, we get
\[
 \lim\limits_{N\to \infty}  \frac{\int_0^T  \langle A_h^{1+\alpha-\frac\gamma2} \overline{V}^N, \sum\limits_{k_3=0,1\leq|\bk|\leq N} c_{\bk} \phi_{\bk} dW_{\bk} \rangle}{\int_0^T \|A_h^{1+\frac{\alpha}{2}} \overline{V}^N\|^2 dt} =  0, \;
\lim\limits_{N\to \infty} \frac{\int_0^T  \langle A_z A^{\alpha-\frac\gamma2} \widehat{V}^N, \sum\limits_{1\leq|\bk|\leq N,|\bk'|=\sqrt{q}|k_3|} c_{\bk} \phi_{\bk} dW_{\bk} \rangle}{\int_0^T \|A_z A^{\frac{\alpha}{2}} \widehat{V}^N\|^2 dt}  =0,
\]
and 
\[
\lim\limits_{N\to \infty} \frac{\int_0^T  \langle A_z A^{\alpha-\frac\gamma2} \widetilde{V}^N, \sum\limits_{k_3\neq0,1\leq|\bk|\leq N} c_{\bk} \phi_{\bk} dW_{\bk} \rangle}{\int_0^T \|A_z A^{\frac{\alpha}{2}} \widetilde{V}^N\|^2 dt}  =0,
\]
in probability. Therefore, $\lim\limits_{N\to \infty} \nu^N_{h1} = \nu_h$ in probability follows from the first limit. 
Moreover, thanks to Lemma \ref{lemma:ratio-1} and Lemma \ref{lemma:order-linear}, we have
\begin{align*}
   \lim\limits_{N\to \infty}\frac{ \int_0^T \langle A_h A^{\alpha} \widetilde{V}^N,  A_z  \widetilde{V}^N \rangle dt}{\int_0^T \|A_z A^{\frac{\alpha}{2}} \widetilde{V}^N\|^2 dt} = \lim\limits_{N\to \infty} \frac{ \mathbb E \int_0^T \|A_h^{\frac12} A_z^{\frac12} A^{\frac\alpha2} \widetilde{U}^N \|^2 dt}{\mathbb E\int_0^T \|A_z A^{\frac{\alpha}{2}} \widetilde{U}^N\|^2 dt} = \mathcal O(1).
\end{align*}
Consequently, 
\begin{align*}
   \lim\limits_{N\to \infty} \frac{\int_0^T  \langle A_h^{1+\alpha-\frac\gamma2} \overline{V}^N, \sigma_0\sum\limits_{k_3=0,1\leq|\bk|\leq N} c_{\bk} \phi_{\bk} dW_{\bk} \rangle \int_0^T \langle A_h A^{\alpha} \widetilde{V}^N,  A_z  \widetilde{V}^N \rangle dt}{\int_0^T \|A_h^{1+\frac{\alpha}{2}} \overline{V}^N\|^2 dt\int_0^T \|A_z A^{\frac{\alpha}{2}} \widetilde{V}^N\|^2 dt} = 0 \quad \text{in probability},
\end{align*}
and one deduces that $\lim\limits_{N\to \infty} \nu^N_{z1} = \nu_z$ in probability. The proof for $\lim\limits_{N\to \infty} \widehat{\nu^N_{z1}}$ follows similarly.
\end{proof}

Next, we study the consistency of $\nu_{h2}^N$, $\nu_{h3}^N$, and $\nu_{z2}^N$, $\widehat{\nu_{z2}^N}$,  $\nu_{z3}^N$, $\widehat{\nu_{z3}^N}$. To achieve this, we need to estimate the nonlinear terms and show that they are negligible.

\begin{lemma}\label{lemma:nonlinear-term}
Assume that $\gamma>\frac92$. Suppose that $V$ is the solution to \eqref{PE-system} with an initial condition $V(0) = V_0 \in \mathcal D(A^{\frac12+\gamma'})$ for all $\max\left\{\frac54, \frac\gamma2-1\right\}<\gamma'<\frac\gamma2-\frac34$, and
  $U$ is the solution to \eqref{PE-original-linear} with $U(0)=0$. Then, for any $\alpha > \gamma-2$, we have
  \begin{align*}
      &\lim\limits_{N\to\infty} \frac{\int_0^T  \langle A_h^{1+\alpha} \overline{V}^N , \overline{B_N(V,V)} \rangle dt}{\int_0^T \|A_h^{1+\frac{\alpha}{2}} \overline{V}^N\|^2 dt} = \lim\limits_{N\to\infty} \frac{\int_0^T  \langle A_h^{1+\alpha} \overline{V}^N , \overline{B_N(V^N,V^N)} \rangle dt }{\int_0^T \|A_h^{1+\frac{\alpha}{2}} \overline{V}^N\|^2 dt} = 0 \quad a.s.,
      \\
      &\lim\limits_{N\to\infty}  \frac{\int_0^T  \langle A_z A^{\alpha} \widehat{V}^N, \widehat{B_N(V,V)}  \rangle dt}{\int_0^T \|A_z A^{\frac{\alpha}{2}} \widehat{V}^N\|^2 dt} = \lim\limits_{N\to\infty} 
       \frac{\int_0^T  \langle A_z A^{\alpha} \widehat{V}^N, \widehat{B_N(V^N,V^N)}  \rangle dt}{\int_0^T \|A_z A^{\frac{\alpha}{2}} \widehat{V}^N\|^2 dt} = 0 \quad a.s.,
       \\
      &\lim\limits_{N\to\infty}  \frac{\int_0^T  \langle A_z A^{\alpha} \widetilde{V}^N, \widetilde{B_N(V,V)}  \rangle dt}{\int_0^T \|A_z A^{\frac{\alpha}{2}} \widetilde{V}^N\|^2 dt} = \lim\limits_{N\to\infty} 
       \frac{\int_0^T  \langle A_z A^{\alpha} \widetilde{V}^N, \widetilde{B_N(V^N,V^N)}  \rangle dt}{\int_0^T \|A_z A^{\frac{\alpha}{2}} \widetilde{V}^N\|^2 dt} = 0 \quad a.s..
  \end{align*}
\end{lemma}
\begin{proof}
    By the Cauchy–Schwartz inequality, and Lemma \ref{lemma:order-linear} and Lemma \ref{lemma:ratio-1}, it is enough to show that 
    \begin{align*}
        &\lim\limits_{N\to\infty} N^{-(4+2\alpha-2\gamma)} \int_0^T \|A^{\frac\alpha2} \overline{B_N(V,V)}\|^2 dt = \lim\limits_{N\to\infty} N^{-(4+2\alpha-2\gamma)} \int_0^T \|A^{\frac\alpha2} \overline{B_N(V^N,V^N)}\|^2 dt =0 \quad a.s.,
        \\
        &\lim\limits_{N\to\infty} N^{-(4+2\alpha-2\gamma)} \int_0^T \|A^{\frac\alpha2} \widehat{B_N(V,V)}\|^2 dt = \lim\limits_{N\to\infty} N^{-(4+2\alpha-2\gamma)} \int_0^T \|A^{\frac\alpha2} \widehat{B_N(V^N,V^N)}\|^2 dt =0 \quad a.s.,
        \\
        &\lim\limits_{N\to\infty} N^{-(5+2\alpha-2\gamma)} \int_0^T \|A^{\frac\alpha2} \widetilde{B_N(V,V)}\|^2 dt = \lim\limits_{N\to\infty} N^{-(5+2\alpha-2\gamma)} \int_0^T \|A^{\frac\alpha2} \widetilde{B_N(V^N,V^N)}\|^2 dt =0 \quad a.s..
    \end{align*}
Since $\gamma>\frac92$, it follows that $\frac\alpha2 > \frac\gamma2-1 >\frac54$. The application of Lemma \ref{lemma:a1} yields that 
\begin{align}\label{nonlinear-terms-proof-1}
   \int_0^T \|A^{\frac\alpha2} B_N(V,V)\|^2 dt \leq \int_0^T \|A^{\frac\alpha2} B(V,V)\|^2 dt 
    \leq  C\int_0^T \|A^{\frac\alpha2} V\|^2 \|A^{\frac\alpha2 + \frac12} V\|^2 dt.
\end{align}  
Note that $\frac12+\gamma' > \frac\gamma2-\frac12 > \frac\gamma2 - \frac34$, and by Lemma~\ref{lemma:regularity-V} we infer that for any arbitrarily small $\varepsilon>0$,
\begin{equation}\label{nonlinear-terms-proof-2}
    V\in L_{\text{loc}}^2((0,\infty);\mathcal D(A^{\frac\gamma2-\frac14-\varepsilon})\cap C([0,\infty); \mathcal D(A^{\frac\gamma2-\frac34-\varepsilon})) \quad a.s.. 
\end{equation}

We first consider the case where $\gamma-2<\alpha<\gamma-\frac32$. In this scenario, we have $\frac\alpha2 < \frac\gamma2-\frac34$ and $\frac\alpha2 + \frac12 < \frac\gamma2-\frac14$. As a result, combining \eqref{nonlinear-terms-proof-1} and \eqref{nonlinear-terms-proof-2}  gives $\int_0^T \|A^{\frac\alpha2} B_N(V,V)\|^2 dt<\infty$ a.s.. Then, since all the quantities
$\|A^{\frac\alpha2} \overline{B_N(V,V)}\|$, $\|A^{\frac\alpha2} \overline{B_N(V^N,V^N)}\|$, $\|A^{\frac\alpha2} \widehat{B_N(V,V)}\|$, $\|A^{\frac\alpha2} \widehat{B_N(V^N,V^N)}\|$, $\|A^{\frac\alpha2} \widetilde{B_N(V,V)}\|$, $\|A^{\frac\alpha2} \widetilde{B_N(V^N,V^N)}\|$ are bounded above by $\|A^{\frac\alpha2} B_N(V,V)\|$, the result follows.

On the other hand, when $\alpha\geq \gamma-\frac32$, taking an element $\alpha' \in (\gamma-2, \gamma-\frac32)$, one has
\begin{align*}
   \int_0^T \|A^{\frac\alpha2} B_N(V,V)\|^2 dt \leq N^{2(\alpha-\alpha')} \int_0^T \|A^{\frac{\alpha'}2} B_N(V,V)\|^2 dt 
   \leq N^{2(\alpha-\alpha')} \int_0^T \|A^{\frac{\alpha'}2} V\|^2 \|A^{\frac{\alpha'}2 + \frac12} V\|^2  dt.
\end{align*}
Note that such a choice of $\alpha'$ gives $$\int_0^T \|A^{\frac{\alpha'}2} V\|^2 \|A^{\frac{\alpha'}2 + \frac12} V\|^2  dt<\infty \quad a.s.,$$ and also $$-(4+2\alpha-2\gamma)+2(\alpha-\alpha') = -4-2\alpha'+2\gamma<0.$$ Therefore the result still holds.
\end{proof}
\begin{remark}
Here, it is necessary to impose the condition $\gamma>\frac92$ to control the nonlinear terms. This requirement is stronger than that of the 2D NSE \cite{cialenco2011parameter}. The main reason is that the nonlinear estimates for the $3D$ PE are worse than those for the 2D NSE. 
\end{remark}

\begin{proposition}\label{prop:consistent-2}
    For $\gamma>\frac92$ and $\alpha>\gamma-2$, $\nu^N_{h2}$ and $\nu^N_{h3}$ are weakly consistent estimators of the true parameters $\nu_h$, and $\nu^N_{z2}$ , $\widehat{\nu_{z2}^N}$, $\nu^N_{z3}$, $\widehat{\nu_{z3}^N}$ are weakly consistent estimators of the true parameters $\nu_z$.
\end{proposition}
\begin{proof}
    The proof is analogous to that of Proposition~\ref{prop:consistent-1} and uses the estimates in Lemma~\ref{lemma:nonlinear-term}.
\end{proof}

\subsection{Asymptotic normality: proof of Theorem~\ref{thm:normality}}\label{sec:normality}
We finally address the joint asymptotic normality of $\nu^N_{h1}$ and $\widehat{\nu^N_{z1}}$, and provide the proof of Theorem~\ref{thm:normality}. To this end, we first recall their definitions in  \eqref{nuh1N} and \eqref{hatnuz1N}:
\begin{align*}
 \nu_{h1}^N &= \nu_h - \frac{\int_0^T  \langle A_h^{1+\alpha-\frac\gamma2} \overline{V}^N, \sigma_0\sum\limits_{k_3=0,1\leq|\bk|\leq N} c_{\bk} \phi_{\bk} dW_{\bk} \rangle }{\int_0^T \|A_h^{1+\frac{\alpha}{2}} \overline{V}^N\|^2 dt},\\
\widehat{\nu^N_{z1}} &= \nu_z - \frac{\int_0^T  \langle A_z A^{\alpha-\frac\gamma2} \widehat{V}^N, \sigma_0\sum\limits_{1\leq|\bk|\leq N, |\bk'|=\sqrt{q}|k_3|} c_{\bk} \phi_{\bk} dW_{\bk} \rangle}{\int_0^T \|A_z A^{\frac{\alpha}{2}} \widehat{V}^N\|^2 dt} \\ 
& \quad + q\frac{\int_0^T  \langle A_h^{1+\alpha-\frac\gamma2} \overline{V}^N, \sigma_0\sum\limits_{k_3=0,1\leq|\bk|\leq N} c_{\bk} \phi_{\bk} dW_{\bk} \rangle}{\int_0^T \|A_h^{1+\frac{\alpha}{2}} \overline{V}^N\|^2 dt}.
\end{align*}
Using \eqref{ratio-5}, \eqref{ratio-6} and the fact that $\overline V^N = \overline U^N + \overline R^N$ and $\widehat V^N = \widehat U^N + \widehat R^N$ , it suffices to show that 
\begin{equation}\label{eq:jointnormal}
 N^2\left(
 \begin{array}{l}
 \frac{\int_0^T  \langle A_h^{1+\alpha-\frac\gamma2} \overline{U}^N, \sigma_0\sum\limits_{k_3=0,1\leq|\bk|\leq N} c_{\bk} \phi_{\bk} dW_{\bk} \rangle }{\mathbb E\int_0^T \|A_h^{1+\frac{\alpha}{2}} \overline{U}^N\|^2 dt} := \overline I^N  \\
\frac{\int_0^T  \langle A_z A^{\alpha-\frac\gamma2} \widehat{U}^N, \sigma_0\sum\limits_{1\leq|\bk|\leq N,|\bk'|=\sqrt{q}|k_3|} c_{\bk} \phi_{\bk} dW_{\bk} \rangle }{\mathbb E\int_0^T \|A_z A^{\frac{\alpha}2} \widehat{U}^N\|^2 dt} := \widehat I^N
 \end{array}
 \right) \stackrel{\mathcal{D}}{\longrightarrow} \Xi,
\end{equation}
where $\Xi$ is a two-dimensional normal random variable
\begin{equation}\label{eq:dist}
    \Xi \sim \mathcal{N}\left(\left[\begin{array}{c}
 0\\
 0
 \end{array}\right], 
 \left[\begin{array}{cc}
\frac{2\nu_h}{\pi T} \frac{(2+\alpha-\gamma)^2}{2+2\alpha-2\gamma} & 0 \\
0 & (q+1)\frac{\nu_h + \frac1q\nu_z}{\pi T} \frac{(2+\alpha-\gamma)^2}{2+2\alpha-2\gamma}
\end{array}\right]\right);
\end{equation}
and that
\begin{align}
     &\lim\limits_{N\to\infty} N^2\frac{\int_0^T  \langle A_h^{1+\alpha-\frac\gamma2} \overline{R}^N, \sigma_0\sum\limits_{k_3=0,1\leq|\bk|\leq N} c_{\bk} \phi_{\bk} dW_{\bk} \rangle }{\mathbb E\int_0^T \|A_h^{1+\frac{\alpha}{2}} \overline{U}^N\|^2 dt} = 0, \label{asym-2}\\
    &\lim\limits_{N\to\infty} N^2\frac{\int_0^T  \langle A_zA^{\alpha - \frac\gamma2} \widehat{R}^N, \sigma_0\sum\limits_{k_3=0,1\leq|\bk|\leq N} c_{\bk} \phi_{\bk} dW_{\bk} \rangle }{\mathbb E\int_0^T \|A_z A^{\frac{\alpha}{2}} \widehat{U}^N\|^2 dt} = 0, \text{ in probability}. \label{asym-4}
\end{align}
 
We start by establishing \eqref{eq:jointnormal}--\eqref{eq:dist}. Notably, the two fractions $\overline I^N$ and $\widehat I^N$ in \eqref{eq:jointnormal} are independent since they are driven by different $W_{\bk}$. Consequently, it remains to show that 
\begin{equation}\label{asym-1}
   N^2 \overline I^N \stackrel{\mathcal{D}}\longrightarrow \mathcal{N}\left(0, \frac{2\nu_h}{\pi T} \frac{(2+\alpha-\gamma)^2}{2+2\alpha-2\gamma}\right),
\end{equation}
 and
\begin{equation}\label{asym-3}
  N^2 \widehat I^N \stackrel{\mathcal{D}}\longrightarrow  \mathcal{N}\left(0, (q+1)\frac{\nu_h + \frac1q\nu_z}{\pi T} \frac{(2+\alpha-\gamma)^2}{2+2\alpha-2\gamma}\right). 
\end{equation}
For $\overline I^N$, we define $\sigma_{\bk} = |\bk|^{2+2\alpha-\gamma} \overline{U}_{\bk}$, and $\xi_{\bk}= \int_0^T |\sigma_{\bk}|^2 dt$. From \eqref{order-Ukbar} and \eqref{order-Var} we obtain
\begin{equation}
    \mathbb E[\xi_{\bk}] \sim |\bk|^{4+4\alpha-2\gamma} |\bk|^{-2\gamma-2} = |\bk|^{2+4\alpha-4\gamma}, \quad \Var [\xi_{\bk}] \sim |\bk|^{8+8\alpha-4\gamma} |\bk|^{-4\gamma-6} = |\bk|^{2+8\alpha-8\gamma}.
\end{equation}
Let $\zeta_n =  \sum\limits_{k_3=0, |\bk|^2 = n} \xi_{\bk}$ and $b_n= \sum\limits_{j=1}^n \mathbb E[\zeta_j]$. According to Lemma \ref{lemma:order}, under the condition $2+4\alpha-4\gamma>-2$ (equivalently, $\alpha>\gamma-1$), we have $$b_n \sim \sum\limits_{\bk\in \mathbb Z^2, 1\leq |\bk|\leq \sqrt{n}} |\bk|^{2+4\alpha-4\gamma} \sim n^{2+2\alpha-2\gamma}.$$ This implies that $b_n$ is increasing and unbounded. Consequently,
\[
\sum\limits_{n=1}^\infty \frac{\Var\left[ \zeta_n \right]}{b_n^2}\lesssim \sum\limits_{n=1}^\infty \frac{n^{\frac12} n^{1+4\alpha-4\gamma}}{n^{4+4\alpha-4\gamma}} = \sum\limits_{n=1}^\infty n^{-\frac52} < \infty,
\]
where we have used the fact that $\Big|\{\bk\in\mathbb Z^3, k_3=0: |\bk|^2=n\}\Big| = \Big|\{\bk\in\mathbb Z^2: |\bk|^2=n\}\Big|\lesssim \sqrt n.$ Therefore by Lemma \ref{lemma:LLN} we conclude that
\[
\lim\limits_{N\to \infty} \frac{\sum\limits_{n=1}^N \zeta_n}{ \sum\limits_{n=1}^N \mathbb E \zeta_n} = 1 \quad a.s.,
\]
and applying Lemma \ref{lemma:CLT} yields
\[
\lim\limits_{N\to \infty} \sigma_0 \frac{\int_0^T  \langle A_h^{1+\alpha-\frac\gamma2} \overline{U}^N, \sum\limits_{k_3=0,1\leq|\bk|\leq N} c_{\bk} \phi_{\bk} dW_{\bk} \rangle }{\left(\mathbb E\int_0^T \|A_h^{1+\alpha-\frac\gamma2} \overline{U}^N\|^2 dt \right)^{\frac12}} \stackrel{\mathcal{D}}{=} \mathcal N(0,\sigma_0^2).
\]
Finally, using the estimate from \eqref{order-Ubar}, we deduce 
\[
\frac{\left(\mathbb E\int_0^T \|A_h^{1+\alpha-\frac\gamma2} \overline{U}^N\|^2 dt \right)^{\frac12}}{\mathbb E \int_0^T \|A_h^{1+\frac\alpha2} \overline{U}^N\|^2 dt} \asymp \frac1{\sigma_0}\sqrt{\frac{2\nu_h}{\pi T}} \frac{2+\alpha-\gamma}{\sqrt{2+2\alpha-2\gamma}} \frac{1}{N^2},
\]
thus establishing \eqref{asym-1}.

Regarding $\widehat I^N$, for $\bk\in\mathbb Z^3$ such that $|\bk'|=\sqrt{q}|k_3|\neq 0$, let us define $\sigma_{\bk} = \frac1{q+1}|\bk|^{2+2\alpha-\gamma} \widehat{U}_{\bk}$, and $\xi_{\bk}= \int_0^T |\sigma_{\bk}|^2 dt$. From \eqref{order-Ukhat} and \eqref{order-Vhat} we derive
\begin{equation}
    \mathbb E[\xi_{\bk}] \sim |\bk|^{4+4\alpha-2\gamma} |\bk|^{-2\gamma-2} = |\bk|^{2+4\alpha-4\gamma}, \quad \Var [\xi_{\bk}] \sim |\bk|^{8+8\alpha-4\gamma} |\bk|^{-4\gamma-6} = |\bk|^{2+8\alpha-8\gamma}.
\end{equation}
Define $\zeta_n =  \sum\limits_{|\bk|^2 = n,|\bk'|=\sqrt{q}|k_3|} \xi_{\bk}$ and $b_n= \sum\limits_{j=1}^n \mathbb E[\zeta_j]$. By Lemma \ref{lemma:order}, under the condition $2+4\alpha-4\gamma>-2$ (equivalently, $\alpha>\gamma-1$), we have $$b_n \sim \sum\limits_{\bk\in \mathbb Z^2, 1\leq |\bk|\leq \sqrt{n}} |\bk|^{2+4\alpha-4\gamma} \sim n^{2+2\alpha-2\gamma}.$$ As a result, $b_n$ is increasing and unbounded, and we obtain 
\[
\sum\limits_{n=1}^\infty \frac{\Var\left[ \zeta_n \right]}{b_n^2}\lesssim \sum\limits_{n=1}^\infty \frac{n^{\frac12} n^{1+4\alpha-4\gamma}}{n^{4+4\alpha-4\gamma}} = \sum\limits_{n=1}^\infty n^{-\frac52} < \infty,
\]
where we have used the fact that $\Big|\{\bk\in\mathbb Z^3: |\bk|^2=n, |\bk'|=\sqrt{q}|k_3|\}\Big| = 2\Big|\{\bk'\in\mathbb Z^2: |\bk'|^2=\frac{q}{q+1}n\}\Big|\lesssim \sqrt n.$ Therefore, by Lemma \ref{lemma:LLN} we can assert that
\[
\lim\limits_{N\to \infty} \frac{\sum\limits_{n=1}^N \zeta_n}{ \sum\limits_{n=1}^N \mathbb E \zeta_n} = 1 \quad a.s.,
\]
and applying Lemma \ref{lemma:CLT} produces
\[
\lim\limits_{N\to \infty} \sigma_0 \frac{\int_0^T  \langle A_z A^{\alpha-\frac\gamma2} \widehat{U}^N, \sum\limits_{1\leq|\bk|\leq N, |\bk'|=\sqrt{q}|k_3|} c_{\bk} \phi_{\bk} dW_{\bk} \rangle }{\left(\mathbb E\int_0^T \|A_z A^{\alpha-\frac\gamma2} \widehat{U}^N\|^2 dt \right)^{\frac12}} \stackrel{\mathcal{D}}{=} \mathcal N(0,\sigma_0^2).
\]
Finally, using the estimate in \eqref{order-Uhat} and noticing that $A_z = \frac{1}{q+1} A$ when $|k'|=\sqrt{q}|k_3|$, we have
\[
\frac{\left(\mathbb E\int_0^T \|A_z A^{\alpha-\frac\gamma2} \widehat{U}^N\|^2 dt \right)^{\frac12}}{\mathbb E \int_0^T \|A_z A^{\frac\alpha2} \widehat{U}^N\|^2 dt} \asymp \frac1{\sigma_0} \sqrt{q+1}\sqrt{\frac{\nu_h + \frac{1}{q}\nu_z}{\pi T}} \frac{2+\alpha-\gamma}{\sqrt{2+2\alpha-2\gamma}} \frac{1}{N^2},
\]
and \eqref{asym-3} follows. 

We now address the proof of \eqref{asym-2}--\eqref{asym-4}.  As $\alpha>\gamma-1$, we know that $4+2\alpha-2\gamma > 2$. By applying \eqref{ratio-5} and taking $\overline{\delta_2}=2$ in \eqref{residual-to-zero-1}, we get \eqref{asym-2}. 
Similarly, applying \eqref{ratio-6} and taking $\widehat{\delta_2}=2$ in \eqref{residual-to-zero-2}, we obtain \eqref{asym-4}.

\begin{remark}
    We have established the asymptotic normality for the first type of estimators $\nu_{h1}^N$ and $\widehat{\nu_{z1}^N}$. However, proving the asymptotic normality for the second type of estimators $\nu_{h2}^N$ and $\widehat{\nu_{z2}^N}$ is quite challenging and beyond the scope of this work. Indeed, such results remain open even for simpler systems such as 2D NSE (see \cite[Remark 4.9]{cialenco2011parameter}). For some particular equations, asymptotic normality for similar estimators was established in \cite{pasemann2020drift,Pasemann2021}. In the nutshell, a necessary step to show such a result is to revisit Lemma~\ref{lemma:nonlinear-term} and prove a stronger statement, for example,
    \begin{equation*}
        \lim\limits_{N\to\infty} N^2\frac{\int_0^T  \langle A_h^{1+\alpha} \overline{V}^N , \overline{B_N(V,V)} - \overline{B_N(V^N,V^N)} \rangle dt}{\int_0^T \|A_h^{1+\frac{\alpha}{2}} \overline{V}^N\|^2 dt} = 0 \quad a.s..
    \end{equation*}
\end{remark}

\section*{Acknowledgment}
The authors are grateful to the editors and the anonymous referees for their helpful comments, suggestions, and insightful questions which helped to improve significantly the paper.
I.C. acknowledges support from the NSF grant DMS-1907568. R.H. was partially supported by the NSF grant DMS-1953035 and DMS-2420988, the Regents' Junior Faculty Fellowship at UCSB, a grant from the Simons Foundation (MP-TSM-00002783), and the ONR grant under \#N00014-24-1-2432. Q.L. was partially supported by a AMS-Simons Travel Grant. Part of this research was performed while the first two authors were visiting the Institute for Mathematical and Statistical Innovation (IMSI), which is supported by the National Science Foundation.

\appendix

\section{Counting and infinite series}
\begin{lemma}\label{counting}
For $n\in\mathbb N$, we have
\[
\left| \{ \bk\in \mathbb Z^3: |\bk|^2 = n\}\right| \lesssim n.
\]
\end{lemma}
\begin{proof}
Denote $r_3(n) = \left| \{ \bk\in \mathbb Z^3: |\bk|^2 = n\}\right|$ and $r_2(n) = \left| \{ \bk\in \mathbb Z^2: |\bk|^2 = n\}\right|$.
First, consider $n$ of the form $n=4^a(8b+7)$ for some nonnegative integers $a$ and $b$. According to Legendre's three-square theorem \cite{hardy1979introduction}, we know that $r_3(n)=0$.
For other $n\in\mathbb N$, we can compute
\begin{align*}
    r_3(n) = \sum\limits_{k_3 \in \mathbb Z, 1\leq |k_3|^2 \leq n} r_2(n-k_3^2) \leq \sum\limits_{i=0}^n r_2(i) \sim n,
\end{align*}
where the last approximation follows from the Gauss circle problem \cite{hardy1979introduction}.
\end{proof}

\begin{lemma}\label{lemma:order}
For $N\gg1$ and $\alpha > -d$, we have asymptotically that
$$\sum\limits_{\bk\in\mathbb Z^d, 1\leq |\bk|\leq N} |\bk|^\alpha \asymp 
\begin{cases}
    &\frac{(d-1)^{-\frac\alpha d} \omega_d}{\frac\alpha d+1}N^{d+\alpha}, \quad d\geq 2
    \\
    & 2\frac{N^{\alpha+1}}{\alpha+1},\quad d=1.
\end{cases}
$$
Here $\omega_d$ is the volume of the $d$-dimensional unit ball. 
\end{lemma}
\begin{proof}
The case $d=1$ is straightforward by symmetry. For $d\geq 2$,
we first define
\[
\{\lambda_m\}_{m=1,2,...} = \{|\bk|^2\}_{\bk\in\mathbb Z^d\backslash \{0\}}, \quad \lambda_m \text{ is nondecreasing}.
\]
Denote by $C_{\lambda}:=\left|\left\{ m \in \mathbb N: \lambda_m \leq \lambda \right\} \right|.$ Then, from \cite[pp.43]{constantin1988navier} we know that
 $$C_\lambda = (d-1) \left|\left\{ \bk\in \mathbb Z^d\backslash \{0\}: |\bk|\leq \sqrt{\lambda}\right\} \right|.$$ From \cite[Proposition 4.14]{constantin1988navier} one has
\begin{itemize}
    \item $C_{\lambda} \asymp  (d-1)\omega_d\lambda^{\frac d2}$ when $\lambda\gg1$.
    \item $\lambda_j \asymp  ((d-1)\omega_d)^{-\frac2d} j^{\frac2d}$ when $j\gg1.$
\end{itemize}
Then for $N\gg1$,
\begin{align*}
    \sum\limits_{\bk\in\mathbb Z^d, 1\leq |\bk|\leq N} |\bk|^\alpha &\asymp \sum\limits_{m=1}^{\frac{C_{N^2}}{(d-1)}} \lambda_m^{\frac\alpha2} \asymp ((d-1)\omega_d)^{-\frac\alpha d}\sum\limits_{m=1}^{\omega_d N^d} m^{\frac\alpha{d}} \asymp ((d-1)\omega_d)^{-\frac\alpha d}\frac{(\omega_dN^d)^{(1+\frac\alpha d)}}{\frac\alpha d+1}
    \\
    &\asymp \frac{(d-1)^{-\frac\alpha d} \omega_d}{\frac\alpha d+1}N^{d+\alpha}.
\end{align*}
\end{proof}

\section{Estimates of nonlinear terms}

\begin{lemma}\label{lemma:a1}
    Given smooth periodic functions $f$, $g$, and $h$ such that $f$ and $g$ have zero means over $\mathbb T^3$, then, for any $r\geq 0$ and $\varepsilon>0$,
    \begin{equation*}\label{lemma-type1-inequality}
    \begin{split}
        \left|\left\langle A^r  (fg),   h  \right\rangle\right| \leq  C_r\left( \|A^{r}  f\| \|A^{\frac34+\varepsilon}  g\|  + \|A^{\frac34+\varepsilon}  f\| \|A^{r}  g\| \right) \|  h\|.
    \end{split}
\end{equation*}
In particular, this implies that
    \begin{equation*}
        \|A^r(fg)\|\leq C_r\left( \|A^{r}  f\| \|A^{\frac34+\varepsilon}  g\|  + \|A^{\frac34+\varepsilon}  f\| \|A^{r}  g\| \right).
    \end{equation*}
\end{lemma}
\begin{proof}
    We write the Fourier representations of $f, g$ and $h$ as
    \begin{eqnarray*}
    &&\hskip-.8in
     f(\boldsymbol{x}) = \sum\limits_{\boldsymbol{j}\in  \mathbb{Z}^3} \hat{f}_{\boldsymbol{j}} e^{ i\boldsymbol{j}\cdot \boldsymbol{x}}, \quad
    g(\boldsymbol{x}) = \sum\limits_{\boldsymbol{k}\in \mathbb{Z}^3} \hat{g}_{\boldsymbol{k}} e^{ i\boldsymbol{k}\cdot \boldsymbol{x}}, \quad
    h(\boldsymbol{x}) = \sum\limits_{\boldsymbol{l}\in \mathbb{Z}^3} \hat{h}_{\boldsymbol{l}} e^{ i\boldsymbol{l}\cdot \boldsymbol{x}}. 
    \end{eqnarray*}
    It follows that 
    \begin{eqnarray*}
    \left|\left\langle A^r  (fg),  h  \right\rangle\right| = \left|\left\langle fg, A^r  h  \right\rangle\right| \leq  \sum\limits_{\boldsymbol{j}+\boldsymbol{k}+\boldsymbol{l}=0} |\hat{f}_{\boldsymbol{j}}||\hat{g}_{\boldsymbol{k}}||\boldsymbol{l}|^{2r}  |\hat{h}_{\boldsymbol{l}}|.
    \end{eqnarray*}
From $|\boldsymbol{l}| = |\boldsymbol{j}+\boldsymbol{k}| \leq |\boldsymbol{j}|+|\boldsymbol{k}|$, we have 
$
    |\boldsymbol{l}|^{2r} \leq (|\boldsymbol{j}|+|\boldsymbol{k}|)^{2r} \leq C_r(|\boldsymbol{j}|^{2r} + |\boldsymbol{k}|^{2r}),
$
thus
\begin{eqnarray*}
\left|\left\langle A^r  (fg),  h  \right\rangle\right| \leq  \sum\limits_{\boldsymbol{j}+\boldsymbol{k}+\boldsymbol{l}=0} C_r (|\boldsymbol{j}|^{2r}+|\boldsymbol{k}|^{2r})|\hat{f}_{\boldsymbol{j}}||\hat{g}_{\boldsymbol{k}}||\hat{h}_{\boldsymbol{l}}| := A_1 + A_2.
\end{eqnarray*}
With $g$ having zero mean, thanks to the Cauchy–Schwarz inequality, for any $\varepsilon>0$, we have
\begin{eqnarray*}
&&\hskip-.8in 
A_1 = \sum\limits_{\boldsymbol{j}+\boldsymbol{k}+\boldsymbol{l}=0} C_r |\boldsymbol{j}|^{2r}   |\hat{f}_{\boldsymbol{j}}| |\hat{g}_{\boldsymbol{k}}| |\hat{h}_{\boldsymbol{l}}| 
= C_r \sum\limits_{\substack{\boldsymbol{k}\in \mathbb{Z}^3 \\ \boldsymbol{k}\neq 0} }  |\hat{g}_{\boldsymbol{k}}|  \sum\limits_{\substack{\boldsymbol{j}\in \mathbb{Z}^3 \\ \boldsymbol{j}\neq 0} } |\boldsymbol{j}|^{2r} |\hat{f}_{\boldsymbol{j}}|  |\hat{h}_{-\boldsymbol{j}-\boldsymbol{k}}| \nonumber\\
&&\hskip-.58in 
\leq C_r \Big( \sum\limits_{\substack{\boldsymbol{k}\in \mathbb{Z}^3 \\ \boldsymbol{k}\neq 0} } |\boldsymbol{k}|^{-3-4\varepsilon}\Big)^{\frac{1}{2}} \Big( \sum\limits_{\substack{\boldsymbol{k}\in \mathbb{Z}^3 \\ \boldsymbol{k}\neq 0} } |\boldsymbol{k}|^{3+4\varepsilon}  |\hat{g}_{\boldsymbol{k}}|^2\Big)^{\frac{1}{2}}   \sup\limits_{\boldsymbol{k}\in \mathbb{Z}^3}\Big( \sum\limits_{\substack{\boldsymbol{j}\in \mathbb{Z}^3 \\ \boldsymbol{j}\neq 0} } |\boldsymbol{j}|^{4r} |\hat{f}_{\boldsymbol{j}}|^2\Big)^{\frac{1}{2}} \Big( \sum\limits_{\substack{\boldsymbol{j}\in \mathbb{Z}^3 \\ \boldsymbol{j}\neq 0} } |\hat{h}_{-\boldsymbol{j}-\boldsymbol{k}}|^2\Big)^{\frac{1}{2}} \nonumber\\
&&\hskip-.58in \leq C_r \|A^{r}  f\| \|A^{\frac34+\varepsilon}  g\| \|  h\|.
\end{eqnarray*}
For $A_2$, since $f$ has zero mean, similarly we have
\begin{eqnarray*}
&&\hskip-.8in 
A_2 = \sum\limits_{\boldsymbol{j}+\boldsymbol{k}+\boldsymbol{l}=0} C_r |\boldsymbol{k}|^{2r} |\hat{f}_{\boldsymbol{j}}| |\hat{g}_{\boldsymbol{k}}| |\hat{h}_{\boldsymbol{l}}| 
\leq C_r \|A^{\frac34+\varepsilon}  f\| \|A^{r}  g\| \|h\|.
\end{eqnarray*}
\end{proof}

\begin{lemma}(\cite[Lemma 3.1]{petcu2009some})\label{lemma:a2}
  For $\nabla_h \cdot \overline f=0$, one has
  \begin{align*}
     &\| f\cdot \nabla_h g \|\leq C\|f\|_{H^{1}} \|g\|_{H^1} ^{\frac12} \|g\|_{H^{2}}^{\frac12},
     \\
     &\| w(f)\partial_z g\|\leq C\| f\|_{H^{1}}^{\frac12} \| f\|_{H^{2}}^{\frac12} \|g\|_{H^{1}}^{\frac12} \|g\|_{H^{2}}^{\frac12}.
\end{align*}
\end{lemma}

\section{Limit Theorems from Stochastic Analysis}
We recall the following law of large numbers (LLN) and the central limit theorem (CLT) from \cite{cialenco2011parameter}. These will be used in Section~\ref{sec:proof} to prove the consistency and asymptotic normality of the proposed parameters. 

\begin{lemma}[The Law of Large Numbers, {\cite[Lemma~2.2]{cialenco2011parameter}}]\label{lemma:LLN}
	Let $(\xi_n)_{n \geq 1}$ be a sequence of random variables and $(b_n)_{n \geq 1}$ be an increasing sequence of positive numbers such that $\lim _{n \rightarrow \infty} b_n=+\infty$, and
	$$
	\sum_{n=1}^{\infty} \frac{\operatorname{Var} \xi_n}{b_n^2}<\infty.
	$$
	(i) If the random variables $(\xi_n)_{n \geq 1}$ are independent, then
	$$
	\lim _{n \rightarrow \infty} \frac{\sum_{k=1}^n\left(\xi_k-\mathbb{E} \xi_k\right)}{b_n}=0 \quad a.s..
	$$
	(ii) If $(\xi_n)_{n \geq 1}$ are merely uncorrelated random variables, then
	$$
	\lim _{n \rightarrow \infty} \frac{\sum_{k=1}^n\left(\xi_k-\mathbb{E} \xi_k\right)}{b_n}=0 \quad \text{in probability}.
	$$
\end{lemma}

\begin{lemma}[CLT for Stochastic Integrals, {\cite[Lemma~2.3]{cialenco2011parameter}}]\label{lemma:CLT}
	Let $\mathcal S = (\Omega,\mathcal F,\mathbb P, \{\mathcal F_t\}_{t\geq 0}, \{W_{\bk}\})$ be a stochastic basis. Suppose that $\sigma_{\bk} \in L^2(\Omega; L^2([0,T]))$ is a sequence of real-valued predictable processes such that
	\[
	\lim\limits_{N\to \infty} \frac{\sum\limits_{\bk\in\mathbb Z^d, 1\leq |\bk|\leq N} \int_0^T |\sigma_{\bk}|^2 dt}{\sum\limits_{\bk\in\mathbb Z^d, 1\leq |\bk|\leq N} \Eb\int_0^T |\sigma_{\bk}|^2 dt} = 1 \quad \text{in probability}. 
	\]
	Then 
	\[
	\lim\limits_{N\to \infty} \frac{\sum\limits_{\bk\in\mathbb Z^d, 1\leq |\bk|\leq N} \int_0^T \sigma_{\bk} dW_{\bk}}{\left(\sum\limits_{\bk\in\mathbb Z^d, 1\leq |\bk|\leq N} \Eb\int_0^T |\sigma_{\bk}|^2 dt\right)^{1/2}}
	\]
	converges in distribution to a standard normal random variable as $N \to \infty$.
\end{lemma}

\bibliographystyle{alpha}
\bibliography{Reference4}
\end{document}